\documentclass[a4paper,notitlepage,cs4size,cap,indent,oneside,11pt]{article}
\usepackage{graphicx,afterpage}
\usepackage{bm}\usepackage{soul}
\usepackage{empheq}\usepackage{mathtools}
\usepackage[top=0.6in, bottom=0.8in, left=0.6in, right=0.6in]{geometry}
\usepackage{mathrsfs,color}
\usepackage{amsfonts}\usepackage{multirow}
\usepackage{amssymb}\usepackage{caption,comment}
\usepackage{amsmath}\usepackage{float}
\usepackage{amssymb,amsthm}
\usepackage{graphicx,afterpage,cancel}
\usepackage{epstopdf}

\usepackage[pdftex,colorlinks=true]{hyperref}
\hypersetup{CJKbookmarks,%
bookmarksnumbered,%
colorlinks,%
linkcolor=black,%
citecolor=black,%
plainpages=false,%
pdfstartview=FitH}
\numberwithin{equation}{section}
\numberwithin{figure}{section}

\newcommand\tabcaption{\def\@captype{table}\caption}

\newtheorem{theorem}{Theorem}[section]

\newtheorem{proposition}[theorem]{Proposition}

\definecolor{orange}{RGB}{255,127,0}

\def\d{{\, \rm d}}

\afterpage{\clearpage}

\title{A Causation-Based Computationally Efficient Strategy for Deploying Lagrangian Drifters to Improve Real-Time State Estimation}
\author{Erik Bollt\footnote{Department of Mathematics and Department of Electrical and Computer Engineering, Clarkson University, Potsdam, NY 13699-5815}, Nan Chen\footnote{Department of Mathematics, University of Wisconsin-Madison, Madison, WI 53706, USA}~\footnote{Corresponding author; chennan@math.wisc.edu} and Stephen Wiggins \footnote{School of Mathematics, University of Bristol, Bristol, United Kingdom; Department of Mathematics, United States Naval Academy, Annapolis, MD 21402, USA}}
\date{\today}
\begin{document}
\maketitle

\abstract{Deploying Lagrangian drifters that facilitate the state estimation of the underlying flow field within a future time interval is practically important. However, the uncertainty in estimating the flow field prevents using standard deterministic approaches for designing strategies and applying trajectory-wise skill scores to evaluate performance. In this paper an information measurement is developed to quantitatively assess the information gain in the estimated flow field by deploying an additional set of drifters. This information measurement is derived by exploiting causal inference. It is characterized by the inferred probability density function of the flow field, which naturally considers the uncertainty. Although the information measurement is an ideal theoretical metric, using it as the direct cost makes the optimization problem computationally expensive. To this end, an effective surrogate cost function is developed. It is highly efficient to compute while capturing the essential features of the information measurement when solving the optimization problem. Based upon these properties, a practical strategy for deploying drifter observations to improve future state estimation is designed. Due to the forecast uncertainty, the approach exploits the expected value of spatial maps of the surrogate cost associated with different forecast realizations to seek the optimal solution. Numerical experiments justify the effectiveness of the surrogate cost. The proposed strategy significantly outperforms the method by randomly deploying the drifters. It is also shown that, under certain conditions, the drifters determined by the expected surrogate cost remain skillful for the state estimation of a single forecast realization of the flow field as in reality.
}

\section{Introduction}
Lagrangian observations are the drifters that follow a parcel of fluid's movement \cite{griffa2007lagrangian, blunden2019look, honnorat2009lagrangian, salman2008using, castellari2001prediction}. They play an essential role in recovering the multiscale nature underlying the turbulent flow velocity field in both the atmosphere and ocean, which is often not directly measured but can be inferred from the observed tracer trajectories. One systematic way to estimate the underlying flow field is to utilize  Lagrangian data assimilation (DA) \cite{apte2013impact, apte2008data, apte2008bayesian, ide2002lagrangian}, which combines the prior knowledge of the flow field from a given forecast model with the observed drifter trajectories via Bayesian inference. It aims to obtain an optimal statistical estimation of the flow state, known as the posterior distribution. Some of the well-known Lagrangian data sets are the Global Drifter Program \cite{centurioni2017global}, which aims to estimate near-surface currents by tracking the surface drifters deployed throughout the global ocean, and the Argo program \cite{gould2004argo} that utilizes a fleet of robotic instruments drifting with the ocean currents for advancing the operational ocean DA. The drifter observations also provide a powerful tool to facilitate the recovery of the ocean eddies in the Arctic regions, where the sea ice floes play the role of the Lagrangian tracers \cite{mu2018arctic, chen2022efficient}. In addition to the standard ocean drifters, other types of Lagrangian tracers include trash or oil in the ocean \cite{van2012origin, garcia2022structured} and balloons collecting atmospheric data \cite{businger1996balloons}.

Due to their manufacturing cost, only a limited number of drifters are available in many practical applications. Therefore, determining the optimal locations for placing these drifters becomes particularly important to facilitate the recovery of the underlying flow field within a future time interval. In many situations, a certain number of observations already exist, and the goal is to optimally deploy a few additional drifters to obtain a maximum amount of information characterizing the flow field. Determining the preferable locations to discharge Lagrangian tracers fundamentally differs from assigning Eulerian observations. The latter may directly measure the velocity field of the underlying flow, which helps establish optimization problems with possibly fast solvers. Efficient methods to optimally assign Eulerian observations include the data-driven inference \cite{willcox2006unsteady, bui2004aerodynamic, kramer2017sparse, dang2021dmd, drmac2016new, wang2021feasibility}, the sparse optimization \cite{manohar2018data, herzog2015sequentially, chu2021data}, the deep learning \cite{carlberg2019recovering, fukami2019super, erichson2020shallow} and the tensor-based flow reconstruction \cite{farazmand2023tensor}. In contrast, the flow velocity at fixed locations is not the direct output from the drifter observations. Furthermore, the flow velocity is a highly nonlinear function of the drifter displacement \cite{chen2014information, apte2013impact, sun2019lagrangian, salman2008hybrid}, which introduces additional difficulties for mathematical analysis. These facts require new strategies to deploy the Lagrangian observations distinct from those applicable to the standard Eulerian ones. Several theoretical studies \cite{poje2002drifter, treshnikov1986optimal} and numerical algorithms \cite{hernandez1995optimizing, chen2020information, tukan2023efficient} have been developed to guide the drifter deployment by exploiting the geometric properties of the flow field. However, one of the major challenges in practice is that the exact underlying flow field is hard to obtain due to the intrinsic turbulent nature of the underlying flow field. Significant uncertainty appears when the flow field is inferred from a limited number of existing observations via Lagrangian DA. The uncertainty is further amplified when the statistical forecast is applied to predict future states. Therefore, it is of practical importance to consider the uncertainty when designing strategies for deploying Lagrangian drifters that facilitate the state estimation of the flow field within a future time interval.

In this paper, an information measurement is developed to quantitatively assess the information gain in the estimated flow field by deploying an additional set of drifters. This information measurement is derived by exploiting causal inference and is justified from the information-theoretic viewpoint \cite{ash2012information, reza1994introduction}. One unique feature of such an information measurement is that the information gain is characterized by the inferred probability density function (PDF) of the flow field, which is a statistical estimate. This fundamentally differs from the standard approaches by minimizing the point-wise difference between the recovered flow field and the truth or a certain deterministic reference solution. The PDF of the inferred flow state naturally considers the critical role of uncertainty in the state estimation. Specifically, the total information gain is measured by the additional information contained in the posterior distribution from Lagrangian DA or the forecast distribution from ensemble prediction beyond the prior knowledge given by the model equilibrium distribution without exploiting observations. Therefore, the information gain in the inferred flow field is attributed to the drifter observations from the causal viewpoint. It is also worth highlighting that since the model equilibrium distribution is utilized as the reference solution for computing the information gain, the optimization procedure of finding the best locations to deploy drifters based on this information measurement does not require the knowledge of the truth realization of the flow field, which facilitates the information criterion to be applied in practice.

Despite systematically characterizing the information gain and the uncertainty reduction, computational challenges remain in solving the optimization problem by directly treating the information measurement as the cost function. Specifically, evaluating such a cost function involves applying Lagrangian DA to obtain the corresponding posterior distribution. However, when the underlying flow field is modeled by a high-dimensional complex nonlinear system, Lagrangian DA is often computationally expensive and is only affordable for a small number of runs within an allotted time. This prevents using a brute-force search algorithm to find the optimal solution in practice. In addition, as there is no simple closed analytic expression for representing the information measurement as a function of the locations of deploying the drifters, calculating its gradient when applying gradient descent algorithms is also computationally prohibitive. To overcome these computational challenges, an effective surrogate cost function is introduced in this paper. It is highly efficient to compute and applies to high-dimensional complex underlying flow systems while capturing the essential features of the information measurement when solving the optimization problem. This surrogate cost function is based on a nonlinear trajectory diagnostic approach that computes the approximate information gain and the uncertainty reduction along each Lagrangian trajectory. Due to its computational efficiency, a phase portrait map of the surrogate cost can be easily constructed. Each grid point in the map indicates the corresponding surrogate cost when a drifter is placed at that location.

Finally, a practical strategy for deploying the drifter observations in a real-time forecast scenario is developed in this paper. The goal is to seek the optimal locations to discharge drifters at the current time instant that benefits the real-time state estimation of the flow field within a future period. Notably, in addition to the unknown true flow field, uncertainty also appears in forecasting the trajectories of the existing drifters beyond the current time instant. Therefore, applying an ensemble forecast of the flow field and the corresponding drifter trajectories becomes essential. Given each possible realization of the trajectories of the existing drifters, the efficient nonlinear trajectory diagnostic approach provides a corresponding spatial map that indicates the cost of deploying new drifters at different locations. The expected value of such spatial maps associated with different forecast realizations is then utilized to determine the locations for deploying new drifters.

The rest of the paper is organized as follows. Section \ref{Sec:Framework} describes the causation-based information measurement that is used to rigorously quantify the information gain by deploying additional drifters. Section \ref{Sec:Model} describes the stochastic modeling and efficient Lagrangian DA framework.  Section \ref{Sec:Strategy} presents the effective surrogate cost function using the nonlinear trajectory diagnostic approach that efficiently determines the locations for discharging the additional drifters. The numerical experiments are included in Section \ref{Sec:Numerics}. Section \ref{Sec:Conclusion} contains further discussions and conclusions.

\section{A Causation-Based Information Measurement to Quantify the Information Gain by Deploying Additional Drifters}\label{Sec:Framework}
\subsection{State estimation with uncertainty}\label{Subsec:Setup_Framework}
The available information includes a turbulent model characterizing the underlying flow field and the observed trajectories of the existing $L_1$ drifters within the time interval $t\in[0,T]$. The goal is to develop an appropriate strategy for optimally deploying additional $L_2$ drifters that facilitate the state estimation of the flow field.

The model is utilized to create the true flow field, which is one random realization of the system. Note that the true flow field is unknown in practice and is required to estimate. Due to the turbulent nature, two model simulations with slightly different initial conditions or random noise forcing will lead to very different solutions. Therefore, using the model to estimate the true flow field will have a significant uncertainty \cite{majda2016introduction, vallis2017atmospheric}. Since the motion of the drifters is driven by the underlying flow field, combining the model simulations with the observed drifter trajectories will naturally improve the state estimation of the flow field. The drifter trajectories can provide more specific information about the flow field in the nearby locations, which serve as the soft constraints for the flow dynamics to advance the estimation of the flow field in other areas. This is the essence of the Lagrangian DA \cite{apte2013impact, apte2008data, apte2008bayesian, ide2002lagrangian}. Denote by $\mathbf{x}$ the observed drifter trajectories and $\mathbf{u}$ the state of the flow field. DA provides a statistical estimation of the state, which is more appropriate for turbulent systems. Due to the uncertainty in the initial condition and the randomness of the system, the state forecast from the model is given by a PDF $p(\mathbf{u})$, which is called the prior distribution as it uses only the information from the model based on the prior knowledge. The information from observations enters into the state estimation procedure by calculating the likelihood $p(\mathbf{x}|\mathbf{u})$, which narrows down the possible range of the estimated state. The combination of the prior and the likelihood via the Bayesian inference leads to the posterior distribution $p(\mathbf{u}|\mathbf{x})$, which is the solution of the DA:
\begin{equation}\label{Bayesian}
  p(\mathbf{u}|\mathbf{x})\sim p(\mathbf{u})p(\mathbf{x}|\mathbf{u}).
\end{equation}
Due to the additional information in observations, the uncertainty in the posterior distribution is expected to be smaller than that in the prior one $p(\mathbf{u})$. DA facilitates uncertainty reduction in the state estimation of the historical data, namely the reanalysis. It also provides a more accurate initial condition that advances the uncertainty reduction in the forecast state.

Since estimating the flow field using the existing $L_1$ drifters contains uncertainty, such a crucial factor needs to be considered in determining the locations of deploying the additional $L_2$ drifters. In the presence of uncertainty, the standard point-wise measurement is not the most appropriate choice for designing the drifter deploying strategy. Therefore, the pre-requisite of designing a suitable strategy to determine the optimal locations of discharging the additional $L_2$ drifters is to develop a new metric that builds upon the statistical estimates and can systematically quantify the information gain with additional drifters.

\subsection{Causation entropy: measuring the uncertainty reduction using information metrics}\label{Subsec:Information_Gain}

The information gain quantifying the uncertainty reduction in one distribution related to another can naturally be defined as distance-like quantity called a divergence between the two distributions. However, the direct difference between the two PDFs is not a suitable choice to measure the gap between the two statistical estimates. A more systematic way of computing the distance between the two distributions is via the information metrics, which directly compare the statistics and are more appropriate than standard point-wise measurements.

Denote by $\mathbf{x}_{Set_1}:=(\mathbf{x}_{1},\ldots,\mathbf{x}_{L_1})^\mathtt{T}$ the set of the existing $L_1$ drifters and $\mathbf{x}_{Set_2}:=(\mathbf{x}_{L_1+1},\ldots,\mathbf{x}_{L_1+L_2})^\mathtt{T}$ the set of the additional $L_2$ ones, where $\mathbf{x}_l$ should be understood as the trajectory of $\mathbf{x}_l(t)$ from $t=0$ to $t= T$ with $[0,T]$ being the period of observations. Further, denote by $\mathbf{A}$ and $\mathbf{B}$ two random variables, where the associated values are given by $\mathbf{a}$ and $\mathbf{b}$. Let $H(\mathbf{B})$ be the entropy of $\mathbf{B}$ and $H(\mathbf{B}| \mathbf{A})$ the conditional entropy of $\mathbf{B}$ conditioned on the given state of $\mathbf{A}$. They are defined as:
\begin{equation}\label{Definition_Entropy}
  H(\mathbf{B}) = -\int_\mathbf{b} p(\mathbf{b})\log(p(\mathbf{b}))\d \mathbf{b}\qquad\mbox{and}\qquad
  H(\mathbf{B}| \mathbf{A}) = -\int_\mathbf{a}\int_\mathbf{b} p(\mathbf{a},\mathbf{b})\log(p(\mathbf{b}|\mathbf{a}))\d \mathbf{b}\d \mathbf{a}.
\end{equation}
These entropies characterize the uncertainties in the distributions $p(\mathbf{b})$ and $p(\mathbf{b}|\mathbf{a})$ from the information-theoretic viewpoint \cite{majda2006nonlinear}. Since the goal is to make use of the additional $L_2$ drifters to maximize the uncertainty reduction, or equivalently the information gain, in $p(\mathbf{u}|\mathbf{x}_{Set_1},\mathbf{x}_{Set_2})$ related to $p(\mathbf{u}|\mathbf{x}_{Set_1})$, it is natural to compute the difference between the associated conditional entropies.
\begin{equation}\label{Causation_Entropy_Original}
    C^{original}_{\mathbf{x}_{Set_2}\to \mathbf{u}} = H(\mathbf{u}|\mathbf{x}_{Set_1}) - H(\mathbf{u}|\mathbf{x}_{Set_1},\mathbf{x}_{Set_2}).
\end{equation}
This leads to the so-called causation entropy \cite{kim2017causation, almomani2020entropic, fish2021entropic, almomani2020erfit}, which characterizes the additional information due to using the extra $L_2$ drifters.
The causation entropy in \eqref{Causation_Entropy_Original} indicates that, given the existing $L_1$ drifters, if the set of the $L_2$ drifters $\mathbf{x}_{Set_2}$ has no contribution to further reduce the uncertainty in recovering the underlying flow field, then the second term on the right-hand side of \eqref{Causation_Entropy_Original} will equal to the first term. In such a situation, the causation entropy becomes zero. Clearly, a larger value of the causation entropy implies a more significant contribution by imposing the additional $L_2$ drifters.

\subsection{Quantifying the information gain using a modified version of the causation entropy}
The causation entropy in \eqref{Causation_Entropy_Original} provides a first path to study the information gain using a set of additional drifter observations. With the information given by the additional observations, the two distributions $p(\mathbf{u}|\mathbf{x}_{Set_1},\mathbf{x}_{Set_2})$ and $p(\mathbf{u}|\mathbf{x}_{Set_1})$ may become more separated when the difference between the mean of the two distributions of alternative interpretations of the stochastic processes increases. Such a discrepancy evidently indicates the additional information provided by $p(\mathbf{u}|\mathbf{x}_{Set_1},\mathbf{x}_{Set_2})$. However, as each of the conditional entropy $H(p(\mathbf{u}|\mathbf{x}_{Set_1},\mathbf{x}_{Set_2}))$ and $H(p(\mathbf{u}|\mathbf{x}_{Set_1}))$ focuses on describing its intrinsic uncertainty, their difference does not take into account the distance from one distribution to the other \cite{xu2007measuring}. As a consequence, if the two terms on the right-hand side of \eqref{Causation_Entropy_Original} have the same profile but different mean values, then the causation entropy is zero, despite that the additional $L_2$ observations indeed have a contribution to improve the estimation of the underlying flow field.

To overcome such a shortcoming, a slight modification of the original causation entropy is introduced. Instead of first computing the intrinsic uncertainty in the two distributions separately and then taking the difference as in \eqref{Causation_Entropy_Original}, the overall distance between the two distributions is adopted to represent the information gain. This can be achieved by utilizing the relative entropy between these two distributions \cite{majda2010quantifying, majda2005information, kleeman2011information},
\begin{equation}\label{Causation_Entropy_KL_Form}
    C^{modified}_{\mathbf{x}_{Set_2}\to \mathbf{u}} = \int p(\mathbf{u}|\mathbf{x}_{Set_1},\mathbf{x}_{Set_2})\log\frac{p(\mathbf{u}|\mathbf{x}_{Set_1},\mathbf{x}_{Set_2})}{p(\mathbf{u}|\mathbf{x}_{Set_1})}\d \mathbf{u},
  \end{equation}
which is also known as Kullback-Leibler divergence or information divergence \cite{kullback1951information, kullback1987letter, kullback1959statistics}.
Despite the lack of symmetry, the relative entropy has two attractive features. First, $C^{modified}_{\mathbf{x}_{Set_2}\to \mathbf{u}} \geq 0$ with equality if and only if $p(\mathbf{u}|\mathbf{x}_{Set_1},\mathbf{x}_{Set_2})=p(\mathbf{u}|\mathbf{x}_{Set_1})$. Second, $C^{modified}_{\mathbf{x}_{Set_2}\to \mathbf{u}}$ is invariant under general nonlinear changes of variables. These provide an attractive framework for assessing the information gain when the drifters are placed differently. A larger value of $C^{modified}_{\mathbf{x}_{Set_2}\to \mathbf{u}}$ means the additional $L_2$ drifter observations results in a more significant change in $p(\mathbf{u}|\mathbf{x}_{Set_1},\mathbf{x}_{Set_2})$ related to $p(\mathbf{u}|\mathbf{x}_{Set_1})$ based on the existing $L_1$ drifters. As a remark, the information theory can also be utilized to quantify model error, model sensitivity, and prediction skill \cite{majda2010quantifying, majda2011link, majda2012lessons, branicki2012quantifying, branicki2014quantifying, kleeman2011information, kleeman2002measuring, delsole2004predictability, branicki2013non, branstator2010two}.

It is worth highlighting that the eventual goal is to use the $L_1+L_2$ drifters to maximize the information gain in the posterior distribution. Therefore, a more natural comparison is between the posterior distribution $p(\mathbf{u}|\mathbf{x}_{Set_1},\mathbf{x}_{Set_2})$ and the distribution $p_{\infty}(\mathbf{u})$ of the model statistical equilibrium without any observational information. In fact, a large distance between $p(\mathbf{u}|\mathbf{x}_{Set_1},\mathbf{x}_{Set_2})$ and $p(\mathbf{u}|\mathbf{x}_{Set_1})$ does not necessarily indicate $p(\mathbf{u}|\mathbf{x}_{Set_1},\mathbf{x}_{Set_2})$ provides additional information beyond the model equilibrium distribution $p_{\infty}(\mathbf{u})$. An extreme example is that $p(\mathbf{u}|\mathbf{x}_{Set_1},\mathbf{x}_{Set_2})=p_\infty(\mathbf{u})$ and both of them are very distinguishable from $p(\mathbf{u}|\mathbf{x}_{Set_1})$. Therefore, it is more justified to compute directly the information gain in $p(\mathbf{u}|\mathbf{x}_{Set_1},\mathbf{x}_{Set_2})$ beyond $p_{\infty}(\mathbf{u})$, although the searching space of the optimization is only for the locations of the additional $L_2$ drifters,
\begin{equation}\label{Causation_Entropy_KL_Form_new}
    C^{final}_{\mathbf{x}_{Set_2}\to \mathbf{u}} = \int p(\mathbf{u}|\mathbf{x}_{Set_1},\mathbf{x}_{Set_2})\log\frac{p(\mathbf{u}|\mathbf{x}_{Set_1},\mathbf{x}_{Set_2})}{p_{\infty}(\mathbf{u})}\d \mathbf{u}.
  \end{equation}

One practical setup for utilizing the framework of information theory in many applications arises when both the distributions are Gaussian so that $p(\mathbf{u}|\mathbf{x}_{Set_1},\mathbf{x}_{Set_2})\sim\mathcal{N}(\bar{\mathbf{u}}_{Set_{1,2}}, \mathbf{R}_{Set_{1,2}})$ and $p_{\infty}(\mathbf{u})\sim\mathcal{N}(\bar{\mathbf{u}}_{m}, \mathbf{R}_{m})$, where the subscript $\cdot_m$ indicates using the model only. In the Gaussian framework, $C^{final}_{\mathbf{x}_{Set_2}\to \mathbf{u}}$ has the following explicit formula \cite{majda2010quantifying, majda2006nonlinear}
\begin{equation}\label{Signal_Dispersion}
\begin{gathered}
  C^{final}_{\mathbf{x}_{Set_2}\to \mathbf{u}} = \left[\frac{1}{2}(\bar{\mathbf{u}}_{Set_{1,2}}-\bar{\mathbf{u}}_{m})^*(\mathbf{R}_{m})^{-1}(\bar{\mathbf{u}}_{Set_{1,2}}-\bar{\mathbf{u}}_{m})\right] \\\qquad\qquad\qquad\qquad\qquad\qquad+ \left[-\frac{1}{2}\log\det(\mathbf{R}_{Set_{1,2}}\mathbf{R}_{m}^{-1}) + \frac{1}{2}(\mbox{tr}(\mathbf{R}_{Set_{1,2}}\mathbf{R}_{m}^{-1})-\mbox{Dim}(\mathbf{u}))\right],
\end{gathered}
\end{equation}
where $\mbox{Dim}(\mathbf{u})$ is the dimension of $\mathbf{u}$.
The first term in brackets in \eqref{Signal_Dispersion} is called `signal', reflecting the information gain in the mean but weighted by the inverse of the model variance, $\mathbf{R}_{m}$, whereas the second term in brackets, called `dispersion', involves only the covariance ratio, $\mathbf{R}_{Set_{1,2}}\mathbf{R}^{-1}_{m}$. The signal and dispersion terms are individually invariant under any (linear) change of variables which maps Gaussian distributions to Gaussians.

\subsection{The optimization problems using the information measurement as the cost function}
The information metric in \eqref{Causation_Entropy_KL_Form_new} can naturally be used as the cost function for the optimization problem of deploying the additional $L_2$ drifters. The optimal solution is given by the drifter discharge locations that correspond to the maximum value of the information metric.

It is worth noting that the time $t$ of $\mathbf{u}$ is not specified in $p(\mathbf{u}|\mathbf{x}_{Set_1},\mathbf{x}_{Set_2})$. If $t\in[0, T]$, then the problem corresponds to the reanalysis situation. If $t>T$, then a forecast scenario is considered. In the latter case, the data assimilation solution $p(\mathbf{u}|\mathbf{x}_{Set_1})$ via the filtering technique at $t=T$ is computed, which serves as the initial value in the given flow model for continuously estimating the future states using both $\mathbf{x}_{Set_1}$ and $\mathbf{x}_{Set_2}$. In both cases, the averaged information gain within a time interval is calculated. Although the focus of this paper is on the real-time estate estimation scenario, it is helpful to summarize the reanalysis situation as well \cite{chen2023launching} to help understand the additional challenge in the real-time case.

\subsubsection{The reanalysis scenario}
In the reanalysis scenario, the goal is to determine the locations of deploying the $L_2$ drifters at time $t=t^*$ such that the information gain in the recovered flow field from Lagrangian DA within an interval $[t^*-\tau, t^*+\tau] \subset [0, T]$ in light of the total of $L_1+L_2$ observed Lagrangian trajectories is maximized. Once the locations of the $L_2$ new drifters at $t^*$ are determined, the associated trajectories within $[t^*-\tau, t^*+\tau]$ can be computed by integrating the governing equations of these drifters forward and backward in time driven by the true flow field. Equivalently, this scenario can be regarded as optimally choosing $L_2$ additional drifters within $t\in[t^*-\tau, t^*+\tau]$ among many candidate drifters covering the entire domain. Note that the only uncertainty in determining the $L_2$ additional drifter locations at time $t^*$ lies in the imperfect estimation of the underlying flow field using the Lagrangian DA with the $L_1$ existing drifters. Once the $L_2$ drifters are placed at time $t^*$, a nonlinear smoother (a typical DA method for reanalysis) \cite{sarkka2013bayesian, evensen2000ensemble} can directly be utilized to obtain the state estimation $p(\mathbf{u}|\mathbf{x}_{Set_1},\mathbf{x}_{Set_2})$ at each time instant within the interval $[t^*-\tau,t^*+\tau]$.  Then \eqref{Causation_Entropy_KL_Form_new} or \eqref{Signal_Dispersion} is utilized to compute the relative entropy that indicates the information gain associated with the locations in deploying the $L_2$ drifters.

\subsubsection{The real-time state estimation scenario}
In the real-time state estimation scenario, the goal is to determine the locations of deploying the additional $L_2$ drifters at time $t=T$ such that the information gain in the real-time state estimation via Lagrangian DA using these $L_1+L_2$ drifters within the future interval $[T, T+\tau]$ is maximized.
Unlike the reanalysis scenario, the true trajectories of the $L_1+L_2$ drifters for estimating the future states within the time interval $[T, T+\tau]$ become unknown. Both the flow field and the Lagrangian trajectories need to be forecasted. The latter introduces a second source of uncertainty.

The calculation of the cost function becomes different in the scenario for improving the real-time future state estimation. First, the existing $L_1$ drifters are used to estimate the underlying flow field up to the current time instant $t=T$, where the state estimation of the flow field $\mathbf{u}(T)$ is given by the filtering solution
\begin{equation}\label{Filtering_Current}
  p(\mathbf{u}(T)|\mathbf{x}_{Set_1} (s\leq T)),
\end{equation}
where $\mathbf{x}_{Set_1} (s\leq T))$ means the trajectories of $\mathbf{x}_{Set_1}$ from $t=0$ to $t=T$. The goal is to deploy the additional $L_2$ drifters at $t=T$. These $L_1+L_2$ drifters are then used to minimize the uncertainty in the real-time state estimation of the flow field within a future time interval $[T, T+\tau]$ via filtering. To achieve this goal, the flow field within the future time interval $[T,T+\tau]$ needs to be predicted. The initial uncertainty from the filtering solution $p(\mathbf{u}(T)|\mathbf{x}_{Set_1} (s\leq T))$ and the turbulent nature of the underlying flow dynamics lead to uncertainty in the predicted flow field. It can usually be characterized by an ensemble of realizations. As the drifters are driven by the underlying flow field, each realization of the flow field from the ensemble simulation corresponds to one possible set of future drifter trajectories. Because of the ensemble of the drifter trajectories, the state estimation of the future flow field is given by a two-step procedure. Denote by $J$ the number of ensemble members. In the first step, Lagrangian DA is applied to find the filtering posterior distribution of the estimated flow field conditioned on each ensemble member of the forecast drifter trajectories. This results in $J$ posterior distributions $p(\mathbf{u}(t)|\mathbf{x}^j_{Set_{1}}(s\leq t), \mathbf{x}^j_{Set_{2}}(s\leq t))$ for $j = 1,\ldots, J$ at each time $t\in[T,T+\tau]$. Then \eqref{Causation_Entropy_KL_Form_new} or \eqref{Signal_Dispersion} is utilized to compute the information gain associated with each posterior distribution. In the second step, the expectation of the information gain over these $J$ values is computed as the final cost, namely
\begin{equation}\label{Causation_Entropy_KL_Form_new_expectation}
    \mathbb{E}_{J}[C^{final}_{\mathbf{x}_{Set_2}\to \mathbf{u}}] = \frac{1}{J}\sum_{j=1}^J\left[\int p(\mathbf{u}|\mathbf{x}^j_{Set_1},\mathbf{x}^j_{Set_2})\log\frac{p(\mathbf{u}|\mathbf{x}^j_{Set_1},\mathbf{x}^j_{Set_2})}{p_{\infty}(\mathbf{u})}\d \mathbf{u}\right].
  \end{equation}

Note that only one forecast ensemble member corresponds to the true future flow field and the associated drifter trajectories. However, using an arbitrary future realization to determine the drifter locations suffers from biases due to the turbulent nature of the system. In contrast, the expected information gain in \eqref{Causation_Entropy_KL_Form_new_expectation} provides a more robust criterion, which leads to a statistically accurate solution. In Section \ref{Sec:Numerics}, the expected information gain in \eqref{Causation_Entropy_KL_Form_new_expectation} will be used to determine the locations of discharging additional drifters. The study will show that the state estimation of a single future realization of the flow field will remain skillful under certain conditions.

\subsection{Computational challenges in the direct optimization}
Since evaluating $p(\mathbf{u}|\mathbf{x}_{Set_1},\mathbf{x}_{Set_2})$ involves running nonlinear Lagrangian DA, there is, in general, no simple explicit formula to compute its gradient that allows using gradient descent-type algorithm for the optimization. In addition, Lagrangian DA is quite expensive for operational forecast systems. It is only computationally affordable to evaluate $p(\mathbf{u}|\mathbf{x}_{Set_1},\mathbf{x}_{Set_2})$ a few times within an allotted time in practice. This prevents using the brute-force (exhaustive) search algorithm as the total number of trials will be $M^{2L_2}$ when a $M\times M$ mesh grid is utilized.

\section{A Computationally Efficient Surrogate Cost Function}\label{Sec:Strategy}

Since evaluating $p(\mathbf{u}|\mathbf{x}_{Set_1},\mathbf{x}_{Set_2})$ and its gradient is computationally expensive, one practical approach to solve the optimization problem is to introduce a surrogate cost function that is computationally efficient while capturing the essential features of the information measurement when solving the optimization problem.

\subsection{Lagrangian descriptor: a computationally efficient trajectory diagnostic approach}
Lagrangian descriptor is a computationally efficient trajectory diagnostic approach, which has been widely used in many different areas \cite{mendoza2010hidden, madrid2009distinguished, lopesino2017theoretical, mancho2013lagrangian}. It can be used as a surrogate cost function to seek the optimal locations for discharging additional drifters. The link between the Lagrangian descriptor and the information measurement will be discussed in the next subsection. This subsection aims to provide the fundamental knowledge of the Lagrangian descriptor.

Denote by $\mathbf{x}=(x,y)^\mathtt{T}$ the two-dimensional displacement and $\mathbf{u}=(u,v)^\mathtt{T}$ the known two-dimensional velocity field.
The general formula of the Lagrangian descriptor is as follows \cite{mancho2013lagrangian, lopesino2017theoretical, garcia2022lagrangian}
\begin{equation}\label{LD_General_Formula}
  \mathcal{L}(\mathbf{x}^*,t^*) = \int_{t^*-\tau_1}^{t^*+\tau_2} F(\mathbf{x}, t) \d t,
\end{equation}
where $F=|\tilde{F}|$ is a scalar field with positive values and $t$ is time. According to \eqref{LD_General_Formula}, $\mathcal{L}$ is the integrated modulus of $\tilde{F}$
along a trajectory from the past $t^*-\tau_1$ to the future $t^*+\tau_2$ that goes through a point $\mathbf{x}^*$ at time $t^*$. This way yields a space- and time-dependent field computed for all $\mathbf{x}^*$ and $t^*$. One commonly used Lagrangian descriptor is by taking $F$ to be the arc length of the path traced by the trajectory. That is,
\begin{equation}\label{LD_VelocityBased}
M_{vel}(\mathbf{x}^*,t^*) = \int_{t^*-\tau_1}^{t^*+\tau_2} \sqrt{\left(\frac{\partial x}{\partial t}\right)^2+\left(\frac{\partial y}{\partial t}\right)^2} \d t = \int_{t^*-\tau_1}^{t^*+\tau_2} \sqrt{u^2+v^2} \d t.
\end{equation}
Once the Lagrangian descriptor is computed, it is usually normalized to its maximum value in space for illustration purposes. Several other Lagrangian descriptors have also been widely used in practice. One is analog to \eqref{LD_VelocityBased} but exploits the vorticity instead of the arc length in defining $F$. Such a vorticity-based Lagrangian descriptor is essential for identifying vortex-like patterns, such as the eddy detection in the ocean. Another Lagrangian descriptor takes the direct difference between $F$ at the current and a former time instant. It is a useful metric to identify the source of a given target, which is crucial for tracing the source of the oil split and many other environmental problems.

\subsection{Developing a surrogate cost function with the help of Lagrangian descriptor}\label{Subsec:Criteria}
The Lagrangian descriptor based on the arc length of the trajectory in \eqref{LD_VelocityBased} will be adopted as the building block in developing the surrogate cost function.

Despite the unknown true underlying flow field, the inferred flow field can be naturally used to compute the Lagrangian descriptor. Notably, the uncertainty in the inferred or forecast states brings about multiple possible realizations of the flow field. In the reanalysis situation, different realizations of the flow field can be effectively sampled from the posterior distribution of the Lagrangian DA \cite{chen2023launching}. In the real-time forecast scenario, the forecast ensemble members serve as possible realizations. To account for such uncertainties, the Lagrangian descriptor is first computed based on each of these potential realizations of the flow field and the corresponding drifter trajectories. Then taking an expectation over all these realizations leads to
\begin{equation}\label{LD_VelocityBased_Expectation}
\mathbb{E}_{\mathbf{x}}[M_{vel}(\mathbf{x}^*,t^*)] = \mathbb{E}_{\mathbf{x}}\left[\int_{t^*-\tau_1}^{t^*+\tau_2} \sqrt{\left(\frac{\partial x}{\partial t}\right)^2+\left(\frac{\partial y}{\partial t}\right)^2} \d t\right] = \mathbb{E}_{\mathbf{x}}\left[\int_{t^*-\tau_1}^{t^*+\tau_2} \sqrt{u^2+v^2} \d t\right].
\end{equation}
A high value of the expected Lagrangian descriptor in \eqref{LD_VelocityBased_Expectation} can be attributed to two reasons. First, the large value is caused by the long distances traveled by the drifters. In such a situation, the drifters can provide sufficient knowledge about the underlying flow field as they collect information at different spatial locations.  As the drifters travel fast, the locations they pass by usually have strong velocities that lead to a large signal-to-noise ratio in identifying the flow field. In other words, the state variables have strong observability, which allows them to be accurately estimated. Second, a large value of $\mathbb{E}_{\mathbf{x}}[M_{vel}(\mathbf{x}^*,t^*)]$ can also be triggered by the large uncertainty in the estimated states. When significant uncertainties appear in the state estimation, many possible realizations of the flow field have strong values, which leads to large values in the corresponding Lagrangian descriptors. When the expectation is taken, these realizations enhance the value of $\mathbb{E}_{\mathbf{x}}[M_{vel}(\mathbf{x}^*,t^*)]$ in \eqref{LD_VelocityBased_Expectation}. These empirical arguments have been justified using simple analytically solvable examples in \cite{chen2023launching} for the reanalysis scenario. They will be further illustrated in Section \ref{Sec:Numerics}.

The expected Lagrangian descriptor can be linked with the information measurement \eqref{Causation_Entropy_KL_Form_new}. The expected Lagrangian descriptor having a large value corresponds to the locations that either the underlying flow field is strong or the uncertainty is significant. These findings indicate that discharging drifters at those locations will allow the drifters to travel a long distance and carry a large amount of information of the flow field (the former case) or advance the reduction of the local uncertainty (the latter case). In both cases, the information gain is significant. In addition to determining the drifter deployment based on the Lagrangian descriptor, it is also desirable to place the drifters at locations that are separate from each other. This will prevent the drifters from carrying out similar information if their trajectories nearly overlap.

With the above justifications, it is natural to adopt the expected Lagrangian descriptor in \eqref{LD_VelocityBased_Expectation} as the surrogate cost function in solving the optimization problem. Due to its computational efficiency, a phase portrait map of this surrogate cost can easily be constructed. Each grid point in the map indicates the corresponding surrogate cost when a drifter is placed at that location. By excluding the areas around the existing drifters, the maximum of the remaining map corresponds to the optimal solution.

Finally, it is worth remarking that the surrogate cost function via the expected Lagrangian descriptor is only used in solving the optimization problem. The original information measurement \eqref{Causation_Entropy_KL_Form_new} will be used to compute the exact cost and validate the result.

\subsection{Sequential strategy}\label{Subsec:Sequential_Strategy}
In the reanalysis scenario, the Lagrangian descriptor can be used to determine the optimal location for one new drifter, with which an updated flow field and the associated Lagrangian descriptor can be computed to determine the location for the next drifter. Such a sequential drifter deploying strategy via the greedy algorithm helps improve the results. However, in the real-time state estimation scenario, the newly deployed drifters at time $t=T$ will not help improve state estimation at this instant, as it takes time for the flow estimation to respond to the additional observations. In other words, the newly deployed drifters will not help reduce the forecast uncertainty of the flow field. This implies the sequential drifter deployment strategy is not applicable.

\subsection{The entire procedure of deploying drifter observations}\label{Subsec:Procedure}
The entire procedure of deploying drifter observations is summarized in Figure \ref{Illustration_Figure}. To distinguish the differences between the reanalysis and the real-time state estimation scenarios, the procedures of these two cases are compared in the figure.

\subsubsection{The reanalysis scenario}
In the reanalysis scenario, the existing $L_1$ drifters are used to estimate the underlying flow field via Lagrangian DA based on the smoothing technique. The uncertainty in the state estimation allows us to sample a set of flow trajectories, which is then used to form the expected Lagrangian descriptor. Computing the resulting Lagrangian descriptor map for the surrogate cost function leads to the locations of deploying the $L_2$ new drifters.

To validate the result, the $L_1+L_2$ drifters are used in Lagrangian DA to calculate $p(\mathbf{u}|\mathbf{x}_{Set_1},\mathbf{x}_{Set_2})$, which is then used to compute the information gain via the relative entropy \eqref{Causation_Entropy_KL_Form_new} or \eqref{Signal_Dispersion}.

\subsubsection{The real-time state estimation scenario}
In the real-time state estimation scenario, the existing $L_1$ drifters are first used to estimate the underlying flow field via Lagrangian DA up to the current time instant $t=T$ based on the filtering technique. The conditional distribution $p(\mathbf{u}(T)|\mathbf{x}_{Set_1} (s\leq T))$ is used as the initial value to forecast the underlying flow field via an ensemble forecast method. Specifically, the initial condition of each ensemble member is drawn from $p(\mathbf{u}(T)|\mathbf{x}_{Set_1} (s\leq T))$ and then the flow model is run forward. The corresponding Lagrangian descriptor can be computed for each forecast realization of the flow field. The Lagrangian descriptors for different realizations are then averaged to yield an expected value of the Lagrangian descriptor. This is utilized to determine the locations of discharging new drifters.

To validate the result, an ensemble forecast is adopted to predict the trajectories of the underlying flow field, which is then used to obtain one set of the trajectories of $L_1+L_2$ drifters for each ensemble member. An estimated flow field via Lagrangian DA is computed for each ensemble member using the drifter trajectories, which are used to calculate the associated information gain. Finally, the overall cost is provided by the average information gain.

One nuance in characterizing the uncertainty in the two scenarios worth remarking on is the following. The true flow field that drives the drifter trajectories is unknown in the reanalysis situation. Therefore, the Lagrangian descriptor is naturally computed based on the information from the multiple sampled flow field via Lagrangian DA that accounts for the uncertainty in the state estimation. In contrast, in the real-time state estimation scenario, each forecast realization of the flow field that drives the associated drifter trajectories, which are subsequently used for Lagrangian DA, is known. Although Lagrangian DA can be applied to obtain the time evolution of the posterior estimate of such a realization, based on which a set of sampled trajectories is used to account for the uncertainty, this step is not adopted in the strategy developed here for three reasons. First, this sampling strategy has to be applied to all the ensemble members of the flow realizations. The overall computational cost can be significant due to the multiple runs of Lagrangian DA. Second, unlike the reanalysis situation, each true realization of the flow field is known here. It can be naturally used to force the Lagrangian descriptor field to better characterize the flow structure. Third, the dominant uncertainty in the real-time situation is already contained in the ensemble spread of the forecast flow field. It has been considered in computing the expected value of the Lagrangian descriptor.

\begin{figure}[h]
    \centering\hspace*{0cm}
    \includegraphics[width=14cm]{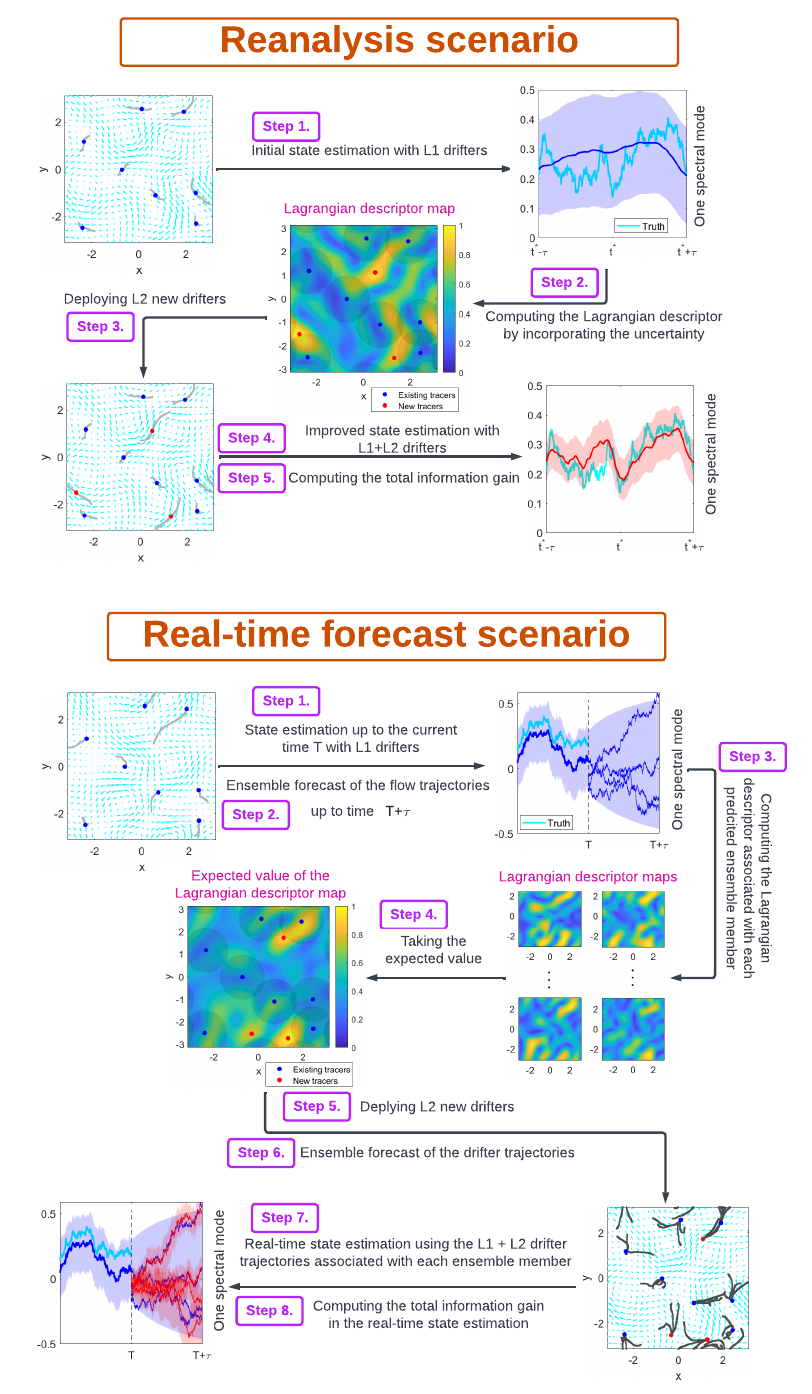}
    \caption{Procedures of deploying drifters in both the reanalysis and real-time forecast scenarios. }
    \label{Illustration_Figure}
\end{figure}
\clearpage

\section{An Efficient Modeling and Lagrangian Data Assimilation Framework}\label{Sec:Model}
All the experiments in this paper are carried out by exploiting an efficient Lagrangian DA framework that contains suitable stochastic forecast models. The unique structure of the stochastic forecast models in such a nonlinear Lagrangian DA framework facilitates the use of analytic formulae for estimating the state of the underlying flow field. With the resulting conditional Gaussian distribution of the flow field, the analytic formula \eqref{Signal_Dispersion} can further be used to compute the information gain. These desirable properties significantly reduce the computational cost and allow us to carry out a brute-force search algorithm to obtain the optimal solution using the exact cost function defined by the causal inference. The results help validate the effectiveness of the proposed approach. Recall that, in practice, the brute-force search algorithm is prohibitively expensive when the standard ensemble DA is applied. Therefore, the computationally efficient strategy developed in this work becomes essential. Notably, the stochastic modeling approach in such an efficient Lagrangian DA framework can provide applicable approximate forecast models that accelerate the Lagrangian DA for many practical problems.

\subsection{Modeling random flow fields}
The underlying turbulent flow model utilized in this work is assumed to be given by a finite summation of spectral modes. Random spectral coefficients are adopted to mimic the intrinsic turbulent features in the flow field, which have been widely used in practice \cite{chen2015noisy, majda2003introduction}. In such a modeling framework, the underlying flow velocity field reads
\begin{equation}\label{Ocean_Velocity}
  \mathbf{u}(\mathbf{x},t) = \sum_{\mathbf{k}\in\mathcal{K},\alpha\in\mathcal{A}}\hat{u}_{\mathbf{k},\alpha}(t)e^{i\mathbf{k}\mathbf{x}}\mathbf{r}_{\mathbf{k},\alpha},
\end{equation}
where $\mathbf{x}=(x,y)^\mathtt{T}$ is the two-dimensional coordinate with a double periodic boundary condition $x,y\in[-\pi,\pi]$. There are two indices for the spectral modes. The index $\mathbf{k}=(k_1,k_2)^\mathtt{T}$ is the wavenumber, and the index $\alpha$ represents the characteristic of the mode, including, for example, the gravity modes and the geophysically balanced modes in the study of many geophysical systems. The set $\mathcal{K}$ consists of the wavenumbers that satisfy $-K_{\mbox{max}}\leq k_1, k_2\leq K_{\mbox{max}}$ with $K_{\mbox{max}}$ being an integer that is pre-determined. The vector $\mathbf{r}_{\mathbf{k},\alpha}$ is the eigenvector, which links the two components of velocity fields $u$ and $v$. For conciseness of notations, the explicit dependence of $\alpha$ on $\hat{u}_{\mathbf{k},\alpha}$ and $\mathbf{r}_{\mathbf{k},\alpha}$ in \eqref{Ocean_Velocity} is omitted in the following discussions. This is a natural simplicity when only one type of characteristic mode is used, which will be the case in the numerical simulations of this work. There, the incompressible flow is considered that includes only the geophysically balanced modes. Therefore, the Fourier coefficient and the eigenvector are written as $\hat{u}_{\mathbf{k}}$ and $\mathbf{r}_{\mathbf{k}}$. Since the left-hand side of \eqref{Ocean_Velocity} is evaluated at physical space, the Fourier coefficient $\hat{u}_{-\mathbf{k}}$ and the eigenvector $\mathbf{r}_{-\mathbf{k}}$ are the complex conjugate of $\hat{u}_{\mathbf{k}}$ and $\mathbf{r}_{\mathbf{k}}$, respectively, for all $\mathbf{k}$. It is worth noting that the framework is not limited to the Fourier basis. Other basis functions and boundary conditions can be utilized in \eqref{Ocean_Velocity} for various applications in practice. Therefore, the representation in \eqref{Ocean_Velocity} is general.

Stochastic models are used to describe the time evolution of each Fourier coefficient $\hat{u}_{\mathbf{k}}$ in \eqref{Ocean_Velocity}, which is a computationally efficient way to mimic the turbulent flow features. Among different stochastic models, the linear stochastic model, namely the complex Ornstein-Uhlenbeck (OU) process \cite{gardiner1985handbook}, is a widely used choice:
\begin{equation}\label{OU_process}
  \frac{\d\hat{u}_{\mathbf{k}}}{\d t} = (- d_\mathbf{k} + i\omega_\mathbf{k}) \hat{u}_{\mathbf{k}} + f_\mathbf{k}(t) + \sigma_\mathbf{k}\dot{W}_\mathbf{k},
\end{equation}
where $d_\mathbf{k}, \omega_\mathbf{k}$ and $f_\mathbf{k}(t)$ are damping, phase and deterministic forcing, $\sigma_\mathbf{k}$ is the noise coefficient and $\dot{W}_\mathbf{k}$ is a white noise. The constants $d_\mathbf{k}$, $\omega_\mathbf{k}$ and $\sigma_\mathbf{k}$ are real-valued while the forcings are complex. The stochastic noise in the linear stochastic model is utilized to effectively parameterize the nonlinear deterministic time evolution of chaotic or turbulent dynamics \cite{majda2016introduction, farrell1993stochastic, berner2017stochastic, branicki2018accuracy, majda2018model, li2020predictability, harlim2008filtering, kang2012filtering} such that different Fourier coefficients (excluding those complex conjugate pairs) are independent with each other. This significantly reduces the computational cost as the operations on the summation of complicated nonlinear terms are replaced by a single stochastic term. The decoupled equations for different modes also accelerate the model forecast. These features are particularly useful for efficient data assimilation since the forecast focuses on the statistics instead of the precise value of each single trajectory. Note that the decoupling between different spectral modes does not break the spatial dependence between the state variables at different grid points in physical space, which is automatically recovered after the spatial reconstruction in light of all the spectral modes.

The mathematical framework of modeling random flow fields in \eqref{Ocean_Velocity}--\eqref{OU_process} has been widely applied to studying turbulent flows. Examples include modeling the rotating shallow water equation \cite{chen2015noisy} and the quasi-geostrophic equation \cite{chen2023stochastic}. Such a framework has also been utilized as an effective surrogate forecast model in data assimilation to recover the flow fields associated with the Navier-Stokes equations \cite{branicki2018accuracy}, the moisture-coupled tropical waves \cite{harlim2013test} and a nonlinear topographic barotropic model \cite{chen2023uncertainty}. In addition, the framework has been used to quantify the uncertainty in geophysical turbulent flows \cite{branicki2013non, chen2023uncertainty, chen2016model}.

\subsection{Lagrangian data assimilation}
Lagrangian data assimilation exploits the observed moving trajectories from drifters to infer the underlying velocity field \cite{apte2013impact, apte2008data, apte2008bayesian, ide2002lagrangian}. It is a widely used approach for state estimation and prediction in geophysics, climate science, and hydrology \cite{griffa2007lagrangian, blunden2019look, honnorat2009lagrangian, salman2008using, castellari2001prediction}.

The observational process is given by the evolution equation of the Lagrangian trajectory,
\begin{equation}\label{Tracer_eqn}
\frac{\d\mathbf{x}}{\d t} = \mathbf{u}(\mathbf{x},t) + \sigma_\mathbf{x}\mathbf{W}_\mathbf{x},
\end{equation}
where $\mathbf{W}_\mathbf{x}$ is a two-dimensional real-valued white noise representing the observational uncertainty and small-scale perturbations to the observed drifter trajectories while $\sigma_\mathbf{x}$ is the noise coefficient. The velocity field $\mathbf{u}$ in \eqref{Tracer_eqn} is given by \eqref{Ocean_Velocity}, which is a highly nonlinear function of $\mathbf{x}$. Usually, $L$ equations of \eqref{Tracer_eqn} are used in Lagrangian data assimilation, representing the observed trajectories of $L$ Lagrangian drifters.

Define $\mathbf{X}=(\mathbf{x}_1,\ldots,\mathbf{x}_L)^\mathtt{T}$ the collection of the $L$ observed drifter trajectories and $\mathbf{U}=\{\hat{u}_\mathbf{k}\}$ the vector that collects the Fourier coefficients. In light of \eqref{Ocean_Velocity}, \eqref{OU_process} and \eqref{Tracer_eqn}, the Lagrangian data assimilation can be written in the following form:
\begin{subequations}\label{eq:cgns}
\begin{align}
\frac{\d \mathbf{X}(t)}{\d t} &= \mathbf{A}(\mathbf{X}, t) \mathbf{U}(t) + \sigma_\mathbf{x} \dot{\mathbf{W}}_\mathbf{X}(t),\label{eq:cgns_X}\\
\frac{\d \mathbf{U}(t)}{\d t} &=  \mathbf{F}_\mathbf{U} + \boldsymbol{\Lambda} \mathbf{U}(t)  + \boldsymbol{\Sigma}_\mathbf{U} \dot{\mathbf{W}}_\mathbf{U}(t),\label{eq:cgns_U}
\end{align}
\end{subequations}
where $\mathbf{A}(\mathbf{X}, t)$ contains all the Fourier bases and is, therefore, a highly nonlinear function of $\mathbf{X}$. Despite the strong nonlinearity in the observational process \eqref{eq:cgns_X}, analytic solutions are available for the Lagrangian data assimilation when the linear stochastic models are used as the surrogate forecast model \cite{liptser2013statistics}. This facilitates the state estimation and uncertainty quantification \cite{chen2014information}.

\begin{proposition}[Posterior distribution of Lagrangian data assimilation: Filtering]
Given one realization of the drifter trajectories $\mathbf{X}(s\leq t)$, the filtering posterior distribution $p(\mathbf{U}(t)|\mathbf{X}(s\leq t))$ of Lagrangian data assimilation \eqref{eq:cgns} is conditionally Gaussian,
 where the time evolutions of the conditional mean $\boldsymbol\mu$ and the conditional covariance $\bf R$ are given by
\begin{subequations}\label{eq:filter}
\begin{align}
\frac{\d\boldsymbol{\mu}}{\d t} &= \left(\mathbf{F}_\mathbf{U} + \boldsymbol{\Lambda} \boldsymbol{\mu}\right)  + \sigma_\mathbf{x}^{-2}\mathbf{R}\mathbf{A}^\ast\left(\frac{\d \mathbf{X}}{\d t} - \mathbf{A}\boldsymbol{\mu} \right),\label{eq:filter_mu}\\
\frac{\d\mathbf{R}}{\d t} &= \boldsymbol{\Lambda}\mathbf{R} + \mathbf{R}\boldsymbol{\Lambda}^\ast + \boldsymbol{\Sigma}_\mathbf{U}\boldsymbol{\Sigma}_\mathbf{U}^\ast - \sigma_{\mathbf{x}}^{-2}\mathbf{R}\mathbf{A}^\ast\mathbf{A}\mathbf{R},\label{eq:filter_R}
\end{align}
\end{subequations}
with $\cdot^*$ being the complex conjugate transpose.
\end{proposition}
\begin{proof}
The proof can be found in \cite{liptser2013statistics, chen2018conditional}.
\end{proof}

The filtering data assimilation is used in the real-time scenario. For the reanalysis situation, the data assimilation solution is given by the smoothing, the details of which can be found in Appendix.

\section{Numerical Experiments}\label{Sec:Numerics}

\subsection{Comparison of the optimization with the computationally efficient surrogate cost function and the brute-force search}
The experiment in this subsection is based on the reanalysis scenario, which is more straightforward than the real-time forecast situation and is, therefore, more appropriate to be used to explain the effectiveness of the surrogate cost function.

The underlying flow model is given by \eqref{Ocean_Velocity}--\eqref{OU_process} with double periodic boundary conditions. The flow field is assumed incompressible, and no mean background flow is included. The parameters are
\begin{equation}\label{Parameters_eddy_model}
  d_\mathbf{k} = 0.5,\qquad \omega_\mathbf{k}=0,\qquad f_\mathbf{k}=0\qquad \mbox{and}\qquad \sigma_\mathbf{k} = 0.5,
\end{equation}
for all $\mathbf{k}$ such that the flow field has an equipartition of the energy. The range of the spectral modes is within $\mathbf{k}\in[-3,3]^2$ and there are in total $48$ modes. The initial distribution of drifters is uniform, consistent with the statistical equilibrium state \cite{chen2014information}. The time instant $t^*=5$ is chosen for deploying new drifters in both scenarios described in Section \ref{Subsec:Setup_Framework}. Note that since the decorrelation time of all Fourier coefficients is only $2$ time units, the initial value has little impact on the flow field at $t^*=5$. The parameter $\tau =1$ is chosen and the time window $[t^*-\tau, t^*+\tau]$ is adopted.

The number of existing drifters is $L_1=10$ and additional $L_2=4$ drifters are added to the field.

\subsubsection{Comparison of the computational costs}
The experiments are based on a laptop with the Processor: Intel(R) Core(TM) i7-1065G7 CPU 1.30GHz and RAM 16.0GB.

When brute-force search algorithm is carried out, a coarse mesh grid with only $10\times 10$ grid points is used to compute the field of the information gain. It takes in total $3428.6$ seconds ($57.14$ minutes) using the greedy algorithm to discharge $L_2=4$ additional drifters. On average, $850$ seconds is needed in even such a coarse grid to compute once the map of the information gain. The computational cost is proportional to the number of grids. If a $32\times 32$ grid is used as in the Lagrangian descriptor case, it will take about $10$ hours to finish the calculation. In contrast, it takes only $131.83$ seconds ($2.20$ minutes) for computing the field of the Lagrangian descriptor on a $32\times 32$ mesh grid. The comparison is also taken by including a set of random strategy of deploying the drifters. When 100 random trials is included, the total computational cost is 1614 seconds (26.92 minutes) without the distance criterion, and 2264 seconds (37.73 minutes) with the distance criterion. Therefore, the computational time of carrying out the Lagrangian descriptor is equal to roughly 8 (without the distance criterion) or 6 (with the distance criterion) random trials. One run of the brute-force searching algorithm equals to about 270 runs of the Lagrangian descriptor. To summarize:
\begin{equation*}
  2.2 \mbox{minutes} = \mbox{Lagrangian descriptor} = 8~ \mbox{or}~ 6~ \mbox{random trials} =  0.37\%~ \mbox{of the brute-force search}.
\end{equation*}

\subsubsection{Comparison of different methods}

Panel (a) shows the flow map of the information gain using the brute-force search and the greedy algorithm of sequentially determining the locations of the $L_2=4$ drifters. In each subpanel, the color map shows the information gain \eqref{Signal_Dispersion} if the next drifter is placed at that location. It can be seen that the flow map stays in the low values if the new drifter is too close to the existing ones. This is the first justification for imposing the distance criterion. As a second justification, it is seen from the two sets of random experiments in Panel (d) that the information gain stays at a higher level if the distance criterion is used. Next, the $L_1$ drifters are adopted to recover the underlying flow field, which is then used to compute the Lagrangian descriptor as the surrogate cost function. Based on the surrogate cost map, the $L_2=4$ drifters are placed all at once. It is seen that when the Lagrangian descriptor is included as one of the criteria, the resulting information gain is higher than the average value of the random trials, even with the distance criterion. If the minimum instead of maximum of the Lagrangian descriptor is utilized, then the corresponding information gain will be lower by about two units, below the average value of the random trials. If the four drifters are placed sequentially, then the information gain is further enhanced. Here, the sequential method means each time one drifter is placed. Then, all the drifters on the field are used to re-compute the surrogate cost map using the expected Lagrangian descriptor, as was described in Section \ref{Subsec:Sequential_Strategy}. See Figure \ref{LD_Sequential} for the change of the map of the surrogate cost function computed from the Lagrangian descriptor. Applying the sequential strategy, the resulting information gain is more significant than all the $100$ random trials without the distance criterion. Note that the maps from the Lagrangian descriptor using the sequential strategy are similar overall, but the uncertainty triggers some differences. Such differences may slightly change the locations to deploy new drifters, which then helps improve the information gain. Nevertheless, since the maps of the surrogate cost using the all-at-once and sequential strategies do not differ too much, it can be concluded that the all-at-once strategy provides an appropriate suboptimal solution. Finally, it is also remarkable that three drifter locations using the Lagrangian descriptor are close to the selected locations from the brute-force search. This proves using the Lagrangian descriptor as a surrogate optimization cost function.

One interesting finding is that the brute-force search result depends on the spatial resolution. In the experiment here, the smallest scale mode has a wavenumber $\mathbf{k}=(3,3)$. It can be seen that using $3\times 3$ mesh grids, the information gain is even worse than $99$ out of $100$ random trials with the distance criterion. With a $5\times 5$ mesh grid, the information gain is still not as high as the one from the Lagrangian descriptor, which is much cheaper. Notably, in practice, the spatial resolution of the system will be much higher. If the choice of the mesh grid cannot resolve a certain range of scales, then it is questionable if the brute-force search will even provide a useful result, not to mention its prohibitively high computational cost. Note that the computational cost of Lagrangian DA also shoots up when the dimension of the underlying flow system increases.

\begin{figure}[h]
    \centering\hspace*{0cm}
    \includegraphics[width=18cm]{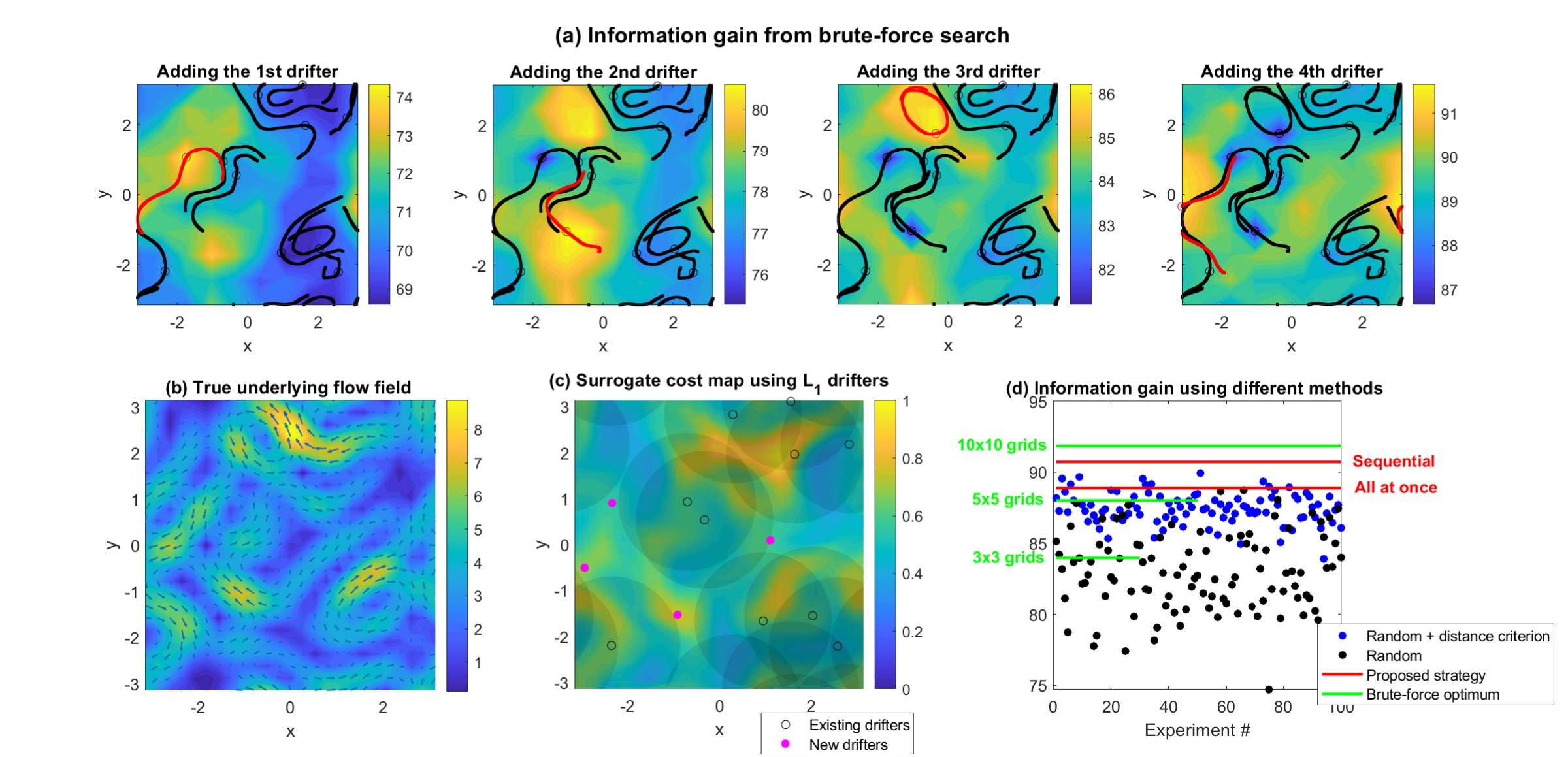}
    \caption{Comparison the information gain of different methods. Panel (a): using brute-force search algorithm to sequentially deploy $L_2=4$ drifters, one at each time. The contour shows the exact information gain computed from the relative entropy \eqref{Signal_Dispersion}. The red dot shows the location of the newly added drifter, which is at the location corresponding to the maximum information gain. Panel (b): the true flow field at $t^*=5$, where the color map shows the amplitude $\sqrt{u^2+v^2}$. Panel (c): the surrogate cost computed from the expected Lagrangian descriptor using the flow field recovered from $L_1$ drifters. The large shading circles indicate the use of distance criterion, where the minimum distance is $1.5$ units here. Panel (d): information gain using different methods. The three green lines represent the information gain using the brute-force search but with different numbers of mesh grids (namely different spatial resolutions). Note that the information gain for the red line and blue dots will have slight shift up- or down-wards if a slight different minimum distance is used as the criterion. But the red line (all-at-once strategy) always stays above most ($\sim90\%$) of the blue dots and the sequential strategy slightly outperforms the all-at-once strategy.  }
    \label{validation_fig}
\end{figure}

\begin{figure}[h]
    \centering\hspace*{0cm}
    \includegraphics[width=18cm]{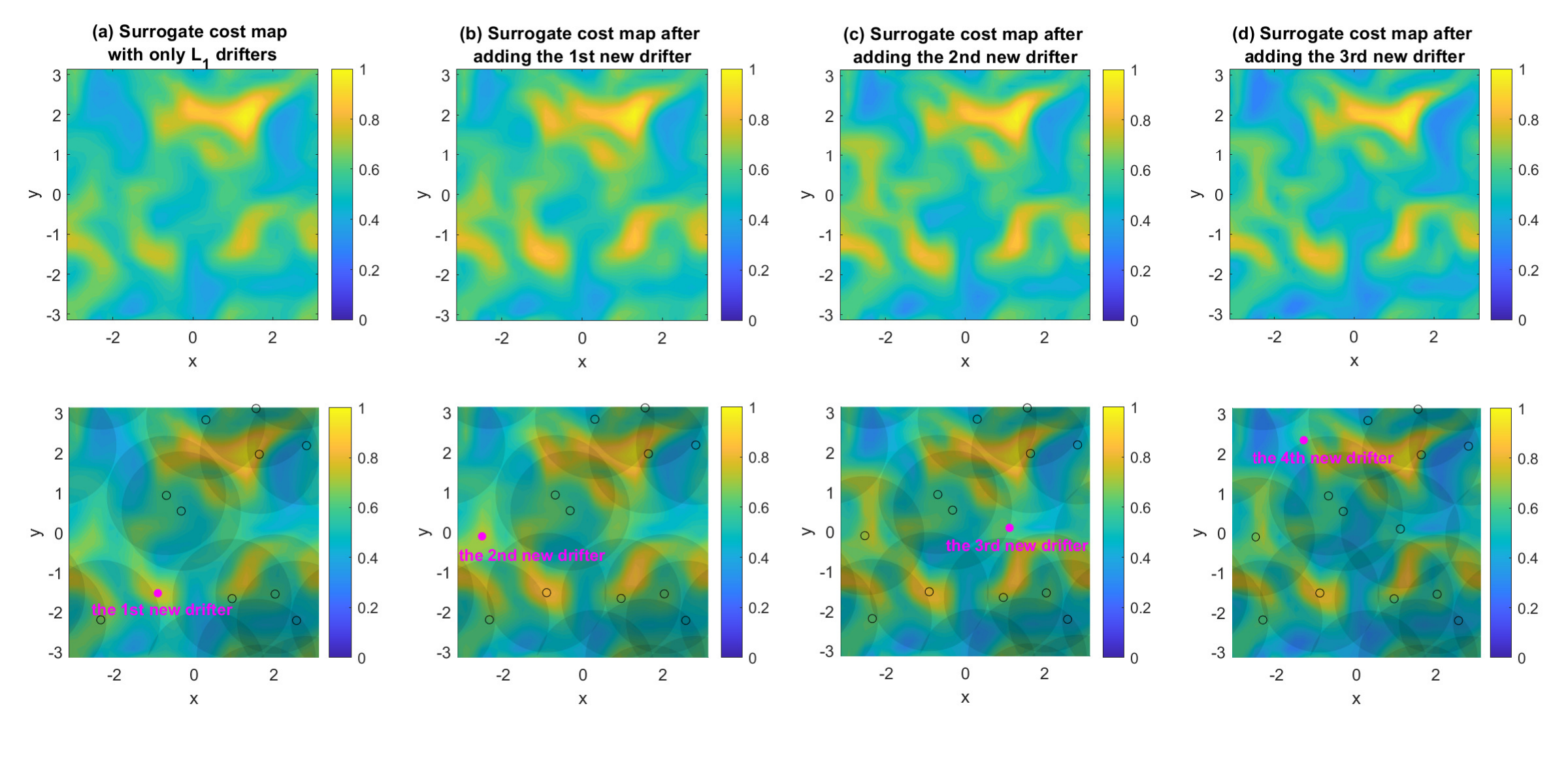}
    \caption{Surrogate cost maps by sequentially deploying the $L_2=4$ drifters. In each column, the top and bottom panels have the same contour plot, representing the surrogate cost. But the bottom panel also includes the locations of the drifters. The existing drifters are represented by black dots. Each time one new drifter is deployed and it is indicated by the red dot. The surrogate cost maps given by the expected Lagrangian descriptor are similar, but the uncertainty triggers some differences. Such differences will change the locations to deploy new drifters, which then helps improve the information gain.   }
    \label{LD_Sequential}
\end{figure}

\subsection{Determining the drifter discharging locations in the real-time state estimation scenario}
\subsubsection{Setup}
Now consider the real-time state estimation scenario. With $\mathbf{k}\in[-3,3]^2$ and in total $48$ modes, the parameters of the underlying dynamics are the same as those in \eqref{Parameters_eddy_model} except $\sigma_\mathbf{k}$, which is reduced to $\sigma_\mathbf{k}=0.125$ such that the trajectories will not travel too far away. There are in total $L_1=10$ drifters existing in the field, which are uniformly distributed in the domain at the initial time instant. Their trajectories from $t=0$ to $t=T=2$ are available. The additional $L_2=4$ drifters are discharged at $t=2$ aiming to improve the state estimation for a future period $t\in[2,2.5]$, which is also the interval used to compute the Lagrangian descriptor. Following the general procedure described in Section \ref{Subsec:Procedure}, an ensemble of the forecast realizations of the flow field and the associated drifter trajectories is computed. The ensemble size is $\mbox{Ens} = 20$. The minimum distance between the new drifter and the existing ones is set to be $1$.

\subsubsection{A numerical experiment}
Panel (a) of Figure \ref{LD_forecast} shows the truth and the estimated time series of mode $(-3,-3)$. The cyan and blue curves from $t=0$ to $t=2$ represent the true signal and the posterior mean from Lagrangian DA. The posterior mean roughly follows the truth, but the uncertainty, indicated by the shading area, is non-negligible. The resulting posterior distribution at $t=2$ is used as the initial value for forecasting the flow model. As is expected, the forecast uncertainty increases from $t=2$ to $t=2.5$. The twenty thin blue and red curves from $t=2$ to $t=2.5$ are the ensemble members of the forecast trajectories and the posterior mean when the $L_2$ additional drifters are placed at the locations determined by the surrogate cost function at $t=2$. The initial value of each random realization is drawn from the posterior distribution at $t=2$, and then the model \eqref{OU_process} is run forward. The initial value of the Lagrangian DA is simply the posterior distribution at $t=2$. One of these realizations is shown in Panel (d), and the associated time evolution of the posterior variance is shown in Panel (e). It is seen that the posterior distribution has a rapid response when the new drifters are placed. Despite the possible initial gap between the posterior mean and the true signal, the posterior mean adjusted quickly towards the posterior mean. The posterior uncertainty also experiences a sharp drop corresponding to the uncertainty reduction with the additional $L_2$ drifters.

Panel (b) of Figure \ref{LD_forecast} shows the map of the surrogate cost function computed from the expected Lagrangian descriptor \eqref{LD_VelocityBased_Expectation}. The blue dots indicate the existing $L_1$ drifters, while the red dots are the locations of the $L_2$ additional drifters at $t=2$ determined by the locations of the maximum values of the expected Lagrangian descriptor and the distance criterion. Panel (c) shows the flow field at $t=2$ and the predicted trajectories of the drifters from $t=2$ to $t=2.5$. These trajectories move in a similar direction but are separate enough, indicating the forecast uncertainty. Among the four newly added drifters, drifter \#1 is placed where the velocity field is the strongest in the domain. This drifter will travel a long distance over the forecast horizon. Drifters \#2 and \#3 are placed where the local flow velocity is relatively strong after excluding the areas close to the existing drifters. In contrast, drifter \#4 is discharged at a location that has a significant uncertainty, which is caused by no drifters in the nearby areas. Therefore, a new drifter placed in such a place will help collect the local information and reduce the overall uncertainty.

Figure \ref{ensemble_members} illustrates the surrogate cost maps associated with each of the twenty realizations of the forecast flow field in the case presented in Figure \ref{LD_forecast}. Despite the apparent discrepancy between different surrogate cost maps due to the uncertainty, these maps share some common features. For example, the surrogate cost around the area discharging the new drifter \#1 is always significant. These coherent structures are crucial for the method to remain skillful for helping the state estimation based on a single forecast realization of the flow field as in reality. Otherwise, the uncertainty will dominate the map, and the underlying features of the flow field will not play a vital role in determining the locations for deploying the new drifters.

\begin{figure}[h]
    \centering\hspace*{0cm}
    \includegraphics[width=18cm]{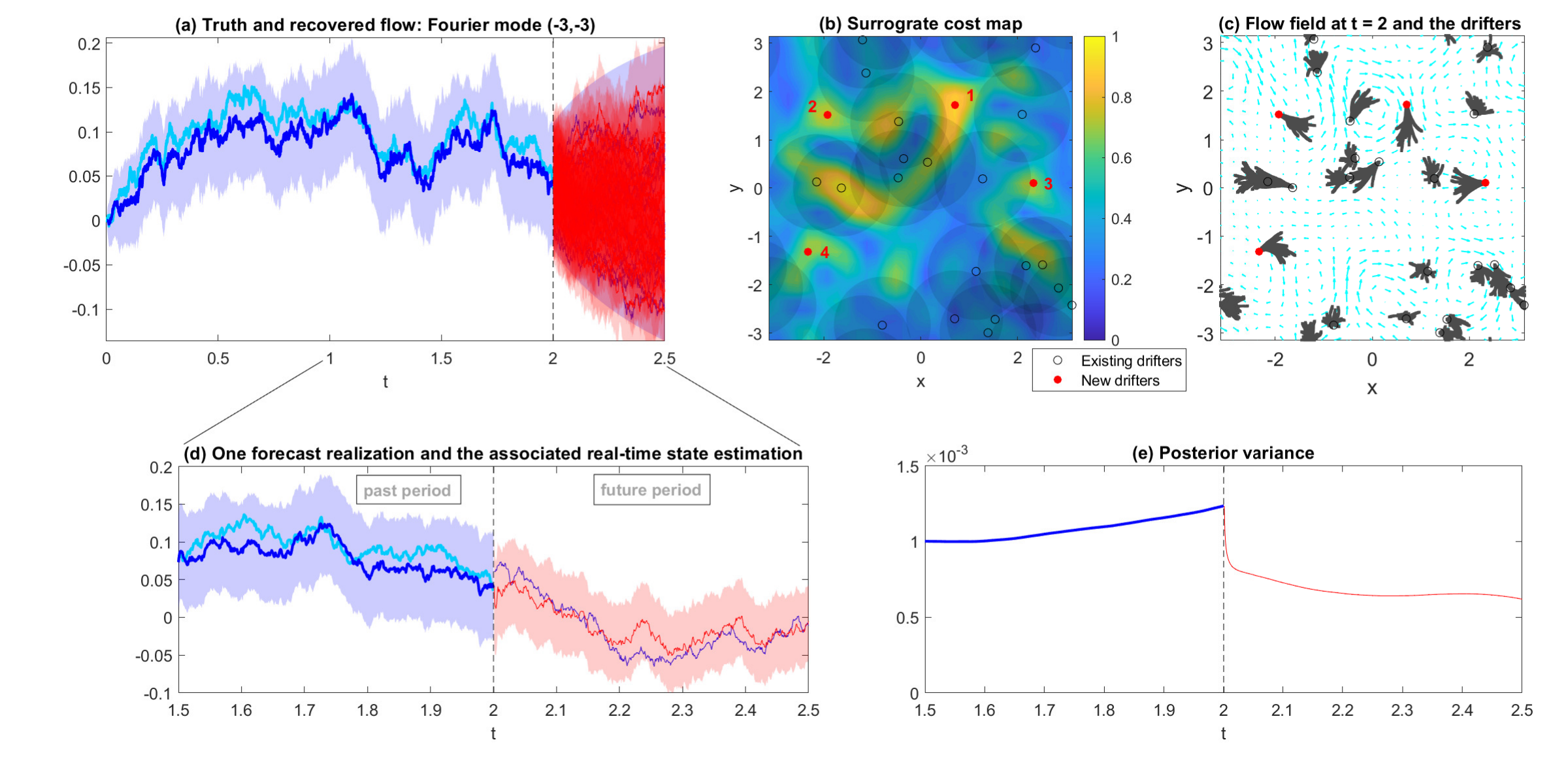}
    \caption{Determining the locations of discharging drifters in the real-time state estimation scenario. The parameters are the same as those in \eqref{Parameters_eddy_model} except $\sigma_\mathbf{k}=0.125$. There are, in total, $L_1=10$ drifters existing in the field. Their trajectories from $t=0$ to $t=T=2$ are available. The additional $L_2=4$ drifters are discharged at $t=2$, aiming to improve the state estimation for a future period $t\in[2,2.5]$. Panel (a): the truth and the estimated time series of mode $(-3,-3)$. The cyan curve from $t=0$ to $t=2$ is the true signal (only the real-part is shown). The blue curve within the same interval is the posterior mean from Lagrangian DA (again only the real-part is shown). The blue shading area shows the two standard deviations from the mean, representing the uncertainty in the Lagrangian DA, where the standard deviation is the square root of the posterior variance of this mode. The uncertainty of the model ensemble forecast from $t=2$ to $t=2.5$ is shown in the blue shading area, where the posterior distribution at $t=2$ serves as the initial condition of running the forecast model \eqref{OU_process}. The twenty thin blue curves from $t=2$ to $t=2.5$ are the ensemble members of the forecast trajectories, and the red curves and the red shading areas are the posterior mean and two standard deviations when the $L_2$ additional drifters are placed at the locations determined by the surrogate cost function at $t=2$. Panels (d) is a zoom-in illustration showing one of the twenty realizations. The associated time evolution of the posterior variance is shown in Panel (e). The contour plot in Panel (b) shows the map of the surrogate cost function computed from the expected Lagrangian descriptor \eqref{LD_VelocityBased_Expectation}. The blue dots indicate the existing $L_1$ drifters, while the red dots are the locations of the $L_2$ additional drifters at $t=2$. Panel (c) shows the flow field at $t=2$ and the predicted trajectories of the drifters from $t=2$ to $t=2.5$.}
    \label{LD_forecast}
\end{figure}

\begin{figure}[h]
    \centering\hspace*{0cm}
    \includegraphics[width=19cm]{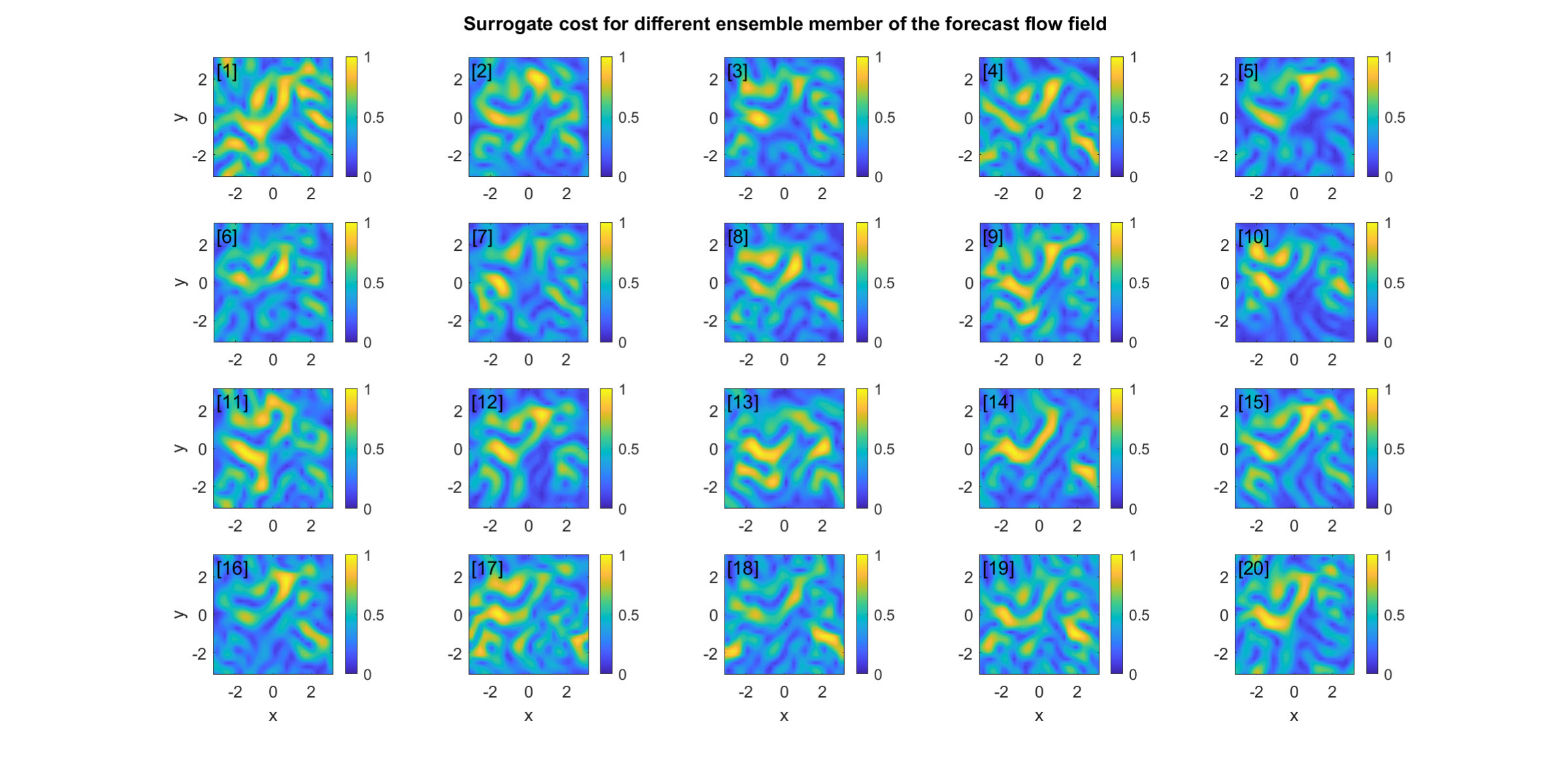}
    \caption{Surrogate cost maps associated with each of the twenty realizations of the forecast flow field in the case presented in Figure \ref{LD_forecast}.  }
    \label{ensemble_members}
\end{figure}

\subsubsection{Statistical analysis}
The previous subsection illustrates the procedure for one single experiment. To systematically study the skill of the proposed strategy, Figure \ref{Comparison_ensemble} includes the results of 100 sets of independent experiments. In each experiment, the true flow field from $t=0$ to $t=2$ is randomly generated. The averaged information gain based on $\mbox{Ens}=20$ forecast realizations when the drifters are deployed using the proposed strategy is shown in red. It is compared with the averaged information gain when the drifters are discharged randomly based on the same number of $\mbox{Ens}=20$ forecast realizations, and the result is shown in one blue dot. Such a test is carried out for 30 independent groups of random assignments, providing 30 blue dots for each experiment. Panel (a) shows the results of the random assignments without considering the distance criterion. The information gain from the proposed strategy outperforms almost all these random trials by a significant amount. Panel (b) includes the results by considering the distance criterion in the random assignments. The improved results from the random trials again justified the necessity of incorporating the distance criterion, which has been validated by the causal analysis. It is also seen that the overall uncertainty in each experiment, namely the difference in the information gain between the random trials, shrinks, which is as expected. Nevertheless, the proposed strategy leads to overall higher information gain than the random assignments, which implies the significance of using the Lagrangian descriptor beyond simply applying the distance criterion.

\begin{figure}[h]
    \centering\hspace*{0cm}
    \includegraphics[width=18cm]{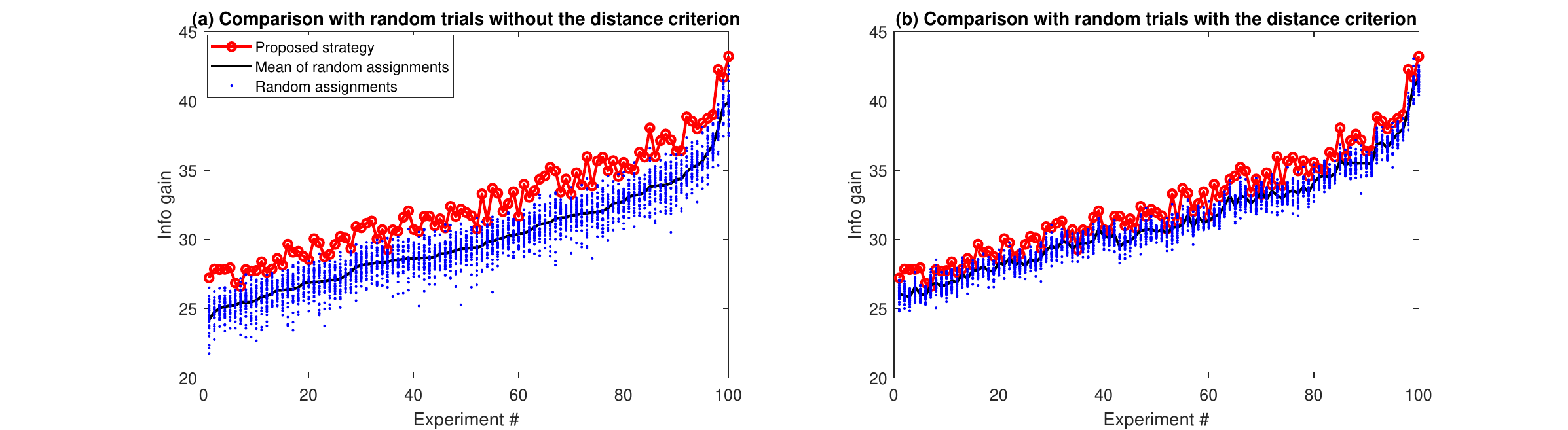}
    \caption{Comparison of the averaged information gain \eqref{Signal_Dispersion} based on $\mbox{Ens}=20$ forecast realizations using the proposed strategy (i.e., by determining using the surrogate cost function) and the one by randomly deploying the drifters. Panel (a) shows the results of the random assignment without considering the distance criterion, while Panel (b) includes the distance criterion. Each panel shows 100 sets of independent experiments. In each experiment, the true flow field from $t=0$ to $t=2$ is randomly generated, which is different from other sets of experiments. The blue dots show the information gain from 30 different random trials of assigning drifters in each experiment. The black color indicates the mean. The red dot shows the information gain using the proposed strategy. For illustration purposes, the experiments in Panel (a) are re-ordered according to the mean of the 30 random trials. The same order is used in Panel (b) such that the trend of the red dots remains the same. }
    \label{Comparison_ensemble}
\end{figure}

It is worth highlighting again that the optimization problem is formed by maximizing the averaged information gain over an ensemble of forecast realizations. The motivation is that the truth of the future realization is unknown, and it is natural to consider the expectation of the information gain. The above study has validated that the proposed strategy succeeds in finding a nearly optimal solution that leads to a large information gain. However, in reality, the truth is given by one single trajectory. Therefore, it is helpful to see whether deploying the drifters at the locations determined by maximizing the expected information gain remains skillful for a random single forecast realization. Figure \ref{Comparison_single} includes such a study. Panel (a) shows that the proposed strategy still outperforms the method by randomly deploying the drifters by an apparent amount overall. Although the advantage of the proposed strategy shrinks when the distance criterion is further included in the random assignments, the proposed strategy still wins 77\% of the experiments in such a tough test; namely, the information gain using the proposed strategy is larger than the mean of those from the random assignments. Note that since the truth of the future is unknown, applying a large number of random assignments based on a single trajectory and selecting the locations that correspond to the maximum of the information gain among these random trials may not necessarily remain informative for the true realization. The proposed strategy provides a systematic and reliable way to determine the locations of deploying the additional drifters.

\begin{figure}[h]
    \centering\hspace*{0cm}
    \includegraphics[width=18cm]{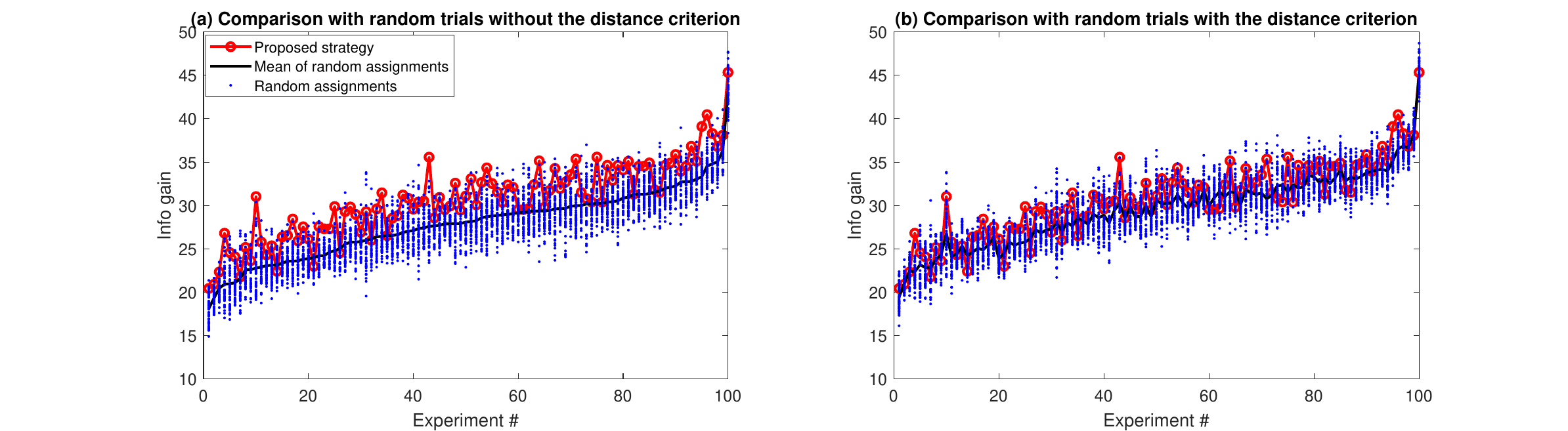}
    \caption{Similar to Figure \ref{Comparison_ensemble} but the information gain is based on only one forecast realization. The number of random trials is also increased to $50$ in each experiment. Note that an ensemble of the forecast realizations is still used to determine the locations for deploying the drifters at $t=2$ in the proposed strategy. This study mimics the realistic situation since the true future trajectory is only one of the realizations of the underlying system. }
    \label{Comparison_single}
\end{figure}

Figure \ref{Comparison_skill_score} shows a skill score of the proposed strategy in comparison with the random trials based on the results shown in Figures \ref{Comparison_ensemble} and \ref{Comparison_single}. Recall that, for each experiment (with one true realization), a number of random tests (30 in Figure \ref{Comparison_ensemble} and 50 in Figure \ref{Comparison_single}) have been carried out. In addition to the mean of each random test, the percentiles can also be computed. In this figure, the x-axis shows different prescribed percentiles of the information gain from the random trials. The y-axis shows the percentage of the random trails with a smaller information gain than the proposed strategy. If the proposed strategy gives an indistinguishable result from the random trials, then the value should be located in the dashed black threshold line. It is seen from the blue and the red curves that the proposed strategy outperforms the mean of the random tests (50\% on the x-axis) in 100\% and 97\% of the experiments without and with the distance criterion in the random assignments, respectively, when comparing the expected information gain. The numbers remain at 98\% and 79\% when comparing the 95 percentiles of the random trials (the last point in the two curves). When focusing on the information gain for a single realization (green and purple curves), the results become worse, as expected. Nevertheless, the green and the purple curves always stay above the dashed black threshold line, implying that the proposed strategy is always more skillful than random assignments, even with the distance criterion. The information gain associated with the proposed strategy in 95\% and 77\% experiments is higher than those using the random assignments without and with the distance criterion. The number remains at 44\% and 18\% when compared with the 95 percentiles of the random experiments.

\begin{figure}[h]
    \centering\hspace*{0cm}
    \includegraphics[width=14cm]{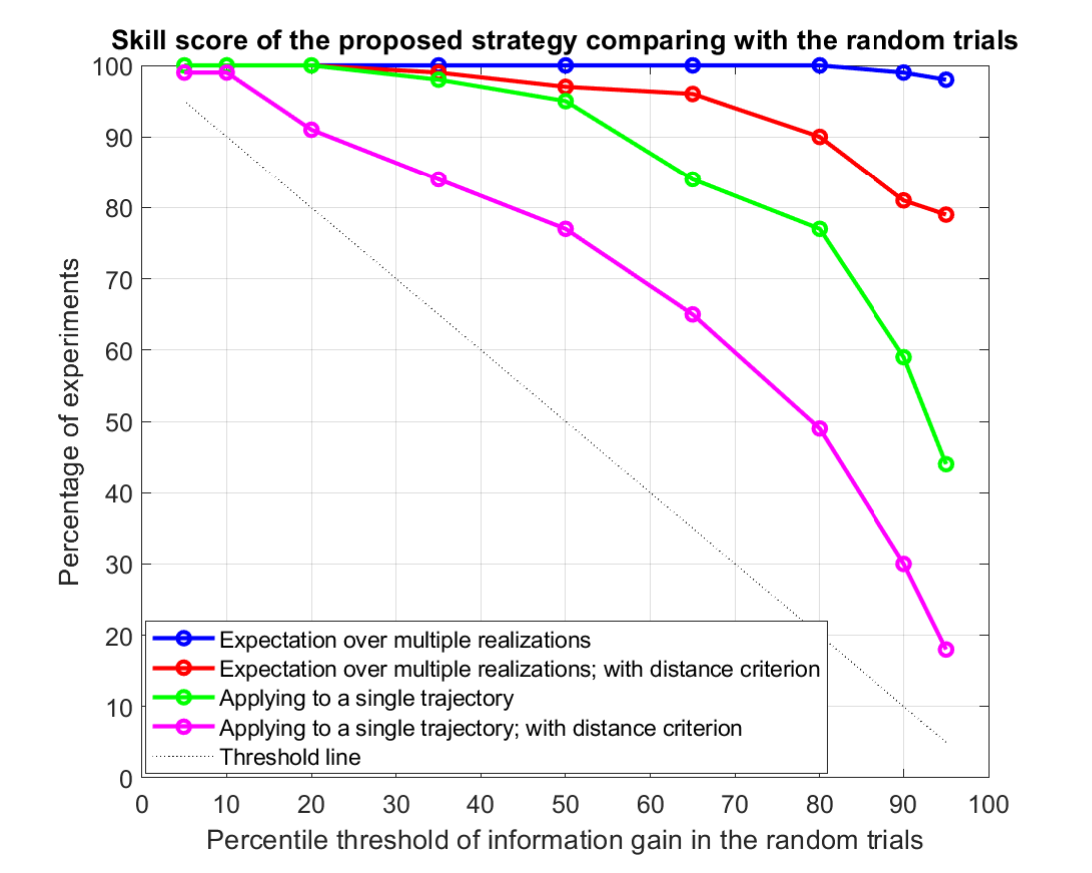}
    \caption{A skill score of the proposed strategy in comparison with the random trials based on the results shown in Figures \ref{Comparison_ensemble} and \ref{Comparison_single}. The x-axis shows different prescribed percentiles of the information gain from the random trials. The y-axis shows the percentage of the random trails with a smaller information gain than the proposed strategy. If the proposed strategy gives an indistinguishable result from the random trials, then the value should be located in the dashed black threshold line.  }
    \label{Comparison_skill_score}
\end{figure}

\section{Discussions and Conclusions}\label{Sec:Conclusion}
\subsection{Sources of uncertainty}
This paper highlights the importance of uncertainty in designing the strategy for discharging Lagrangian drifters. One dominant source of the uncertainty comes from the estimated state using Lagrangian DA. This contrasts with using only the mean of the state estimation, which is usually the output of the standard reanalysis data. If the uncertainty is significant, then the mean time series may lack sufficient information.

The model used in this paper is assumed to be a perfect model, which means the same model is used to generate the true signal, act as the forecast model for Lagrangian DA, and predict the future flow field. Since the model is stochastic (or, in general, turbulent), its realizations with slightly different initial values or random forcing quickly become separate. Thus, statistics are crucial for characterizing the model properties. Regarding the observations, the observational noise level (e.g., the observational noise coefficient $\sigma_{\mathbf{x}}$) is also assumed to be known. In addition, the solution of the Lagrangian DA is exact and given by closed analytic formulae. Therefore, in the reanalysis scenario, the uncertainty solely comes from the state estimation via the Bayesian inference. In the forecast scenario, a second source of uncertainty comes from the intrinsic predictability limit of a turbulent system when the model forecast is carried out.

In applications there are other potential sources of uncertainties. One or some of these uncertainties may appear in a specific situation. Unlike the intrinsic uncertainty from state estimation and forecast that can be handled systematically, as was studied in this paper, these additional uncertainties usually need to be dealt with case by case. First, several additional sources of uncertainties may appear in the models. Since reduced-order models are often adopted in practice to improve computational efficiency, the model error due to coarse-graining, stochastic parameterizations, and imperfect knowledge of physics in these approximate models account for the extra uncertainties \cite{majda2018model}, not only in the DA but also in the forecast. Second, uncertainty may also result from estimating model parameters. In this paper the exact parameters are known. Yet, if the parameter estimation needs to be carried out simultaneously with the DA, then expectation-maximization algorithms have to be applied. The resulting state estimation naturally contains more uncertainty. See the studies in \cite{chen2023uncertainty}. Third, ensemble DA is often needed for more complicated operational systems. Additional uncertainty will naturally come from the sampling error due to insufficient ensemble sizes and the necessary tuning processes, such as the localization and covariance inflation \cite{evensen2000ensemble}. Finally, uncertainty also exists in estimating the observational noise level, which is related to the representation error and may affect the DA skill \cite{janjic2018representation, zeng2018representation}. Studying the effect of these additional uncertainty sources in affecting the drifter deployment can be an interesting topic for future work.

Besides these uncertainties, when the posterior or the model equilibrium distributions are non-Gaussian, applying the Gaussian approximation in \eqref{Signal_Dispersion} for computing the relative entropy may lead to errors. However, since computing the information gain based on the exact formula of the relative entropy in \eqref{Causation_Entropy_KL_Form_new} requires calculating a high-dimensional integration and suffers from the curse of dimensionality, the Gaussian or other approximations that facilitate the evaluation of the information gain become essential. Understanding such approximation errors in determining drifter locations is also practically important.

\subsection{The solvability of the problem}
Recall that the locations for discharging the additional $L_2$ drifters are determined by the map of the surrogate cost function, which is computed based on an ensemble of the forecast realizations of the flow field. In Figure \ref{Comparison_single}, it has been shown that these drifters still provide skillful results for the state estimation of a single forecast realization.

One crucial reason for such a promising result is that the ensemble realizations contain the essential information of a single trajectory. This can be seen from the maps associated with each individual ensemble member shown in Figure \ref{ensemble_members}, which share some common features as the averaged map in Panel (b) of Figure \ref{LD_forecast}. These facts provide valuable evidence to answer a question: When is the strategy determined by the ensemble realizations skillful for a single realization? If the forecast trajectories are fully turbulent and lose predictability, for example, the forecast horizon being much longer than the decorrelation time of the system or the initial uncertainty being too large, then the map of the cost function will be nearly homogeneous, which indicates the uncertainty being equally strong everywhere. In such a situation, the considerable uncertainty makes the problem unsolvable. This means a sufficient number of existing drifters is essential to help reduce the uncertainty and provide helpful guidance for discharging additional drifters.

\subsection{Other application scenarios}
In this paper, the information gain is characterized over the entire flow field. The same procedure can be applied to maximize the gain of the local information by calculating the total cost function map, where the cost is computed based on the pre-determined local area. The strategy can also be applied to maximize the information gain at a specific future time instant instead of the averaged one within a period. In addition, the uncertainty in this work is given by the state estimation using the $L_1$ existing drifters from the Lagrangian DA. Yet, the uncertainty can come from other state estimation methods using either the available Eulerian or Lagrangian observations.

\section*{Acknowledgement}
The research of N.C. is funded by ARO W911NF-23-1-0118. S.W. acknowledges the financial support provided by the EPSRC Grant No. EP/P021123/1 and the support of the William R. Davis '68 Chair in the Department of Mathematics at the United States Naval Academy. The research of E.B. is supported by  the ONR, ARO, DARPA
RSDN, and the NIH and NSF under CRCNS.

\section*{Appendix}
\subsection*{Smoothing and sampling in Lagrangian data assimilation }
With the filtering solution \eqref{eq:filter} in hand, closed analytic formulae are also available for the smoothing solution.

\begin{proposition}[Posterior distribution of Lagrangian data assimilation: Smoothing]\label{Prop:Smoothing}
Given one realization of the drifter trajectories $\mathbf{X}(t)$ for $t\in[0,T]$, the smoother estimate $p(\mathbf{U}(t)|\mathbf{X}(s), s\in[0,T])\sim\mathcal{N}(\boldsymbol\mu_\mathbf{s}(t),\mathbf{R}_\mathbf{s}(t))$ of the coupled system is also Gaussian,
where the conditional mean $\boldsymbol\mu_\mathbf{s}(t)$ and conditional covariance $\mathbf{R}_\mathbf{s}(t)$ of the smoother satisfy the following backward equations
\begin{subequations}\label{Smoother_Main}
\begin{align}
  \frac{\overleftarrow{\d \boldsymbol{\mu}_\mathbf{s}}}{\d t} &=  -\mathbf{F}_\mathbf{U} - \boldsymbol\Lambda\boldsymbol{\mu}_\mathbf{s}  + (\boldsymbol{\Sigma}_\mathbf{U}\boldsymbol{\Sigma}_\mathbf{U}^*)\mathbf{R}^{-1}(\boldsymbol\mu - \boldsymbol{\mu}_\mathbf{s}),\label{Smoother_Main_mu}\\
  \frac{\overleftarrow{\d \mathbf{R}_\mathbf{s}}}{\d t} &= - (\boldsymbol\Lambda + (\boldsymbol{\Sigma}_\mathbf{U}\boldsymbol{\Sigma}_\mathbf{U}^*) \mathbf{R}^{-1})\mathbf{R}_\mathbf{s} - \mathbf{R}_\mathbf{s}(\boldsymbol\Lambda^* + \mathbf{R}^{-1}(\boldsymbol{\Sigma}_\mathbf{U}\boldsymbol{\Sigma}_\mathbf{U}^*))  + \boldsymbol{\Sigma}_\mathbf{U}\boldsymbol{\Sigma}_\mathbf{U}^* ,\label{Smoother_Main_R}
\end{align}
\end{subequations}
with $\boldsymbol\mu$ and $\mathbf{R}$ being given by \eqref{eq:filter}. The notation $\overleftarrow{\d \cdot}/\d t$ corresponds to the negative of the usual derivative, which means that the system \eqref{Smoother_Main} is solved backward over $[0,T]$ with the starting value of the nonlinear smoother being the same as the filter estimate $(\boldsymbol\mu_\mathbf{s}(T), \mathbf{R}_\mathbf{s}(T)) = (\boldsymbol\mu(T), \mathbf{R}(T))$.
\end{proposition}
\begin{proof}
The proof can be found in
 \cite{chen2020learning}.
\end{proof}

The smoother estimate \eqref{Smoother_Main} provides a PDF at each time instant for the recovered velocity field, which includes the uncertainty. Given these PDFs and the temporal dependence, an efficient sampling algorithm of the time series of the velocity field $\mathbf{U}$ from the posterior distributions can be developed. The sampled time series of the velocity field is utilized to compute the Lagrangian descriptor with uncertainty discussed in \cite{chen2023launching} to find the appropriate locations for deploying drifters in the reanalysis scenario.

\begin{proposition}[Sampling trajectories from posterior distributions]\label{Prop:Sampling}
Based on the smoother estimate, an optimal backward sampling of the trajectories associated with the unobserved variable $\mathbf{U}$ satisfies the following explicit formula,
\begin{equation}\label{Sampling_Main}
  \frac{\overleftarrow{\d \mathbf{U}}}{\d t} = \frac{\overleftarrow{\d \boldsymbol\mu_\mathbf{s}}}{\d t} - \big(\boldsymbol\Lambda + (\boldsymbol{\Sigma}_\mathbf{U}\boldsymbol{\Sigma}_\mathbf{U}^*)\mathbf{R}^{-1}\big)(\mathbf{U} - \boldsymbol\mu_\mathbf{s}) + \boldsymbol{\Sigma}_\mathbf{U}\dot{\mathbf{W}}_{\mathbf{U}}(t).
\end{equation}
\end{proposition}
\begin{proof}
The proof can be found in
 \cite{chen2020learning}.
\end{proof}

The temporal dependence in the sampled time series of $\mathbf{U}$ is extremely important. It contains the memory effect of the recovered velocity field, which is a crucial dynamical feature that affects the prediction of the Lagrangian trajectories $\mathbf{X}(t)$. The sampling approach in \eqref{Sampling_Main} fundamentally differs from drawing independent samples at different time instants, giving a noisy time series that lacks the physical properties of $\mathbf{U}$. The sampled trajectories is essential in the Lagrangian descriptor to help determine the location of placing drifters in the reanalysis scenario.

\bibliographystyle{plain}

\end{document}